\documentclass[reqno,11pt]{amsart}
\usepackage{tikz,pgfplots, subcaption, graphicx, amssymb}
\usetikzlibrary{snakes,arrows,patterns}
\newtheorem{theorem}{Theorem}[section]
\newtheorem{corollary}[theorem]{Corollary}
\newtheorem{proposition}[theorem]{Proposition}

\newtheorem{lemma}[theorem]{Lemma}
\newtheorem{definition}[theorem]{Definition}

\theoremstyle{remark}
\newtheorem{remark}[theorem]{Remark}
\newtheorem{example}[theorem]{Example}

\numberwithin{equation}{section}
\allowdisplaybreaks

\newcommand{\defn}[1]{{\emph{#1}}} 
\newcommand{\ta}{\theta}
\newcommand{\C}{\mathbb C}
\newcommand{\N}{\mathbb N}
\newcommand{\naturals}{\mathbb N}
\newcommand{\Z}{\mathbb Z}
\newcommand{\E}{\mathbb E}

\newcommand{\eqdef}{\mbox{\,\raisebox{0.2ex}{\scriptsize\ensuremath{\mathrm:}}\ensuremath{=}\,}} 

\newcommand{\wBinom}[3]{{}_{\stackrel{\phantom w}{\stackrel{\phantom #3}#3}}\!\!
	\begin{bmatrix}#1\\#2\end{bmatrix}} 
\newcommand{\eBinom}[3]{{\begin{bmatrix}#1\\#2\end{bmatrix}}_{#3}} 
\newcommand{\wBinomText}[3]{{}_{\stackrel{\phantom #3}#3}\!
\left[\begin{smallmatrix}#1\\#2\end{smallmatrix}\right]} 

\newcommand{\hSet}[2]{\{#1|#2\}} 
\newcommand{\nSet}[2]{\{#1 \cdots #2\}} 

\definecolor{dblackcolor}{rgb}{0.0,0.0,0.0}
\definecolor{dredcolor}{rgb}{0.9,0.3,0.4}
\definecolor{dbluecolor}{rgb}{0.01,0.02,0.7}
\definecolor{dgreencolor}{rgb}{0.2,0.5,0.0}
\definecolor{dgraycolor}{rgb}{0.30,0.3,0.30}
\definecolor{gr}{rgb}{0.10,0.5,0.20}
\definecolor{myblue}{HTML}{0489B1}
\definecolor{myred}{HTML}{A40000}

\usepackage{hyperref}
\hypersetup{colorlinks=true, citecolor=black, linkcolor=black}

\usepackage{todonotes}

\author[Josef K\"ustner]{Josef K\"ustner$^{*}$}
\address{Fakult\"at f\"ur Mathematik, Universit\"at Wien,
	Oskar-Morgenstern-Platz~1, A-1090 Vienna, Austria}
\email{josef.kuestner@univie.ac.at}
\thanks{$^{*}$ Partially supported by FWF Austrian Science Fund
	grant P32305.}

\author[Michael J.\ Schlosser]{Michael J.\ Schlosser$^{*}$}
\address{Fakult\"at f\"ur Mathematik, Universit\"at Wien,
	Oskar-Morgenstern-Platz~1, A-1090 Vienna, Austria}
\email{michael.schlosser@univie.ac.at}

\author[Meesue Yoo]{Meesue Yoo$^{**}$}
\address{Department of Mathematics, Chungbuk National University,
	Cheongju 28644, South Korea}
\email{meesueyoo@chungbuk.ac.kr}
\thanks{$^{**}$ Partially supported by the National Research Foundation of
	Korea (NRF) grant funded by the Korea government
	No.\! 2020R1F1A1A01064138.}

      \title[Negatively indexed weight-dependent binomial coefficients]{Lattice
        paths and negatively indexed weight-dependent binomial coefficients}


\begin{document}

\begin{abstract}
In 1992, Loeb \cite{L} considered a natural extension of the binomial
coefficients to negative entries and gave a combinatorial interpretation
in terms of hybrid sets. He showed that many of the fundamental properties
of binomial coefficients continue to hold in this extended setting. Recently,
Formichella and Straub \cite{FS} showed that these results can be extended
to the $q$-binomial coefficients with arbitrary integer values and extended
the work of Loeb further by examining arithmetic properties of the
$q$-binomial coefficients. In this paper, we give an alternative combinatorial
interpretation in terms of lattice paths and consider an extension of the more
general weight-dependent binomial coefficients, first defined by the second
author \cite{Schl1}, to arbitrary integer values. Remarkably, many of the
results of Loeb, Formichella and Straub continue to hold in the general
weighted setting. We also examine important special cases of the
weight-dependent binomial coefficients, including ordinary, $q$- and elliptic
binomial coefficients as well as elementary and complete homogeneous symmetric
functions.
\end{abstract}

\maketitle


\section{Introduction}\label{secintro}
Loeb \cite{L} studied a generalization of the binomial coefficients
$\binom{n}{k}=\frac{n!}{k!(n-k)!}$, where $n$ and $k$ are allowed to be
negative integers. He defined them for integers $n$ and $k$ in terms of
\[
  \binom{n}{k} \eqdef \lim_{\epsilon \mapsto 0}
  \frac{\Gamma(n+1+\epsilon)}{\Gamma(k+1+\epsilon)\Gamma(n-k+1+\epsilon)}.
\] 
From this definition it follows that the binomial coefficients with integer
values satisfy the recursion
\begin{subequations}\label{def:bineq}
	\begin{align}
		&\binom{n}{0}=\binom{n}{n} =1
		\qquad\text{for\/ $n\in \mathbb{Z}$},\label{def:bineq1}\\
          \intertext{and for $n,k\in\mathbb{Z}$,
          provided that $(n+1,k)\neq(0,0)$,}
		&\binom{n+1}{k}=
		\binom{n}{k}
		+\binom{n}{k-1},
		 \label{def:bineq2}
	\end{align}
\end{subequations}
and they can be fully characterized by this recursion.
The binomial coefficients with integer values appear in the power series
expansion of $(x+y)^n$ ($n$ being an arbitrary integer) in two ways.
Suppose $f_n(x,y)$ is a function with power series expansions
\begin{subequations}
\begin{align}
  f_n(x,y)&=\sum_{k\geq 0} a_k x^k y^{n-k}
  \quad\; \text{in\; $\C[[x,y,y^{-1}]]$} \\
  \intertext{or}
  f_n(x,y)&=\sum_{k\leq n} b_{k} x^{k} y^{n-k}
  \quad\; \text{in\; $\C[[x,x^{-1},y]]$},
\end{align}
\end{subequations}
then we extract coefficients of the expansions by writing
$[x^ky^{n-k}]f_n(x,y) = a_k$ for $k\geq0$ and $[x^ky^{n-k}]f_n(x,y) = b_k$
for $k<0$ (see \cite{FS}). We have the following extension of the
binomial theorem.
\begin{theorem}[\cite{FS,L}]\label{thm:loeb_binom}
	For $n,k \in \Z$,
	$$[x^ky^{n-k}](x+y)^n = \binom{n}{k}.$$
\end{theorem}
This theorem essentially means that we have the two expansions 
\begin{align*}
(x+y)^n&=\sum_{k\geq 0} \binom{n}{k} x^k y^{n-k}\\
\intertext{and}
(x+y)^n&=\sum_{k \leq n} \binom{n}{k} x^k y^{n-k}.
\end{align*}
If $n\geq 0$, the expansions coincide, but if $n<0$, they are different.
For $n<0$ the first expansion is an expansion in $x$ and $y^{-1}$,
while the second is an expansion in $x^{-1}$ and $y$. 

Recently, Formichella and Straub \cite{FS} extended this theorem to a
$q$-binomial expansion. They considered an extension of the
$q$-binomial coefficients to arbitrary integer values which can be
characterized by the recursive definition
\begin{subequations}\label{def:qbineq}
	\begin{align}
		&\eBinom{n}{0}{q}=\eBinom{n}{n}{q} =1
		\qquad\text{for\/ $n\in \mathbb{Z}$},\label{def:qbineq1}\\
          \intertext{and for $n,k\in\mathbb{Z}$,
          provided that $(n+1,k)\neq(0,0)$,}
		&\eBinom{n+1}{k}{q}=
		\eBinom{n}{k}{q}
		+\eBinom{n}{k-1}{q} q^{n+1-k}.
		 \label{def:qbineq2}
	\end{align}
\end{subequations}
They proved the following generalization of Theorem~\ref{thm:loeb_binom}. 

\begin{theorem}\label{thm:qbinomial}
  Suppose we have $yx=qxy$ for invertible variables $x$ and $y$,
  and an invertible indeterminate $q$. Then, for $n,k\in \Z$,
	\begin{equation}\label{eqn:qbinomial}
		[x^k y^{n-k}](x+y)^n = \eBinom{n}{k}{q}.
	\end{equation}
\end{theorem}
In their paper they gave a combinatorial interpretation of the $q$-binomial
coefficients \eqref{def:qbineq} in terms of hybrid sets. This interpretation
is a $q$-analogue of a model by Loeb for (ordinary) binomial coefficients. 

Theorem~\ref{thm:qbinomial} also extends the noncommutative $q$-binomial
theorem for nonnegative integers $n$ and $k$: Let $\C_q[x,y]$ be the
associative unital algebra over $\C$
generated by noncommutative variables $x$ and $y$, satisfying the relation
\begin{equation}\label{qcomm}
yx=qxy,
\end{equation}
for an indeterminate $q$. Then, the noncommutative $q$-binomial theorem says
that we have, as an identity in $\C_q[x,y]$,
\begin{align}
(x+y)^n=\sum_{k=0}^{n} \eBinom{n}{k}{q} x^k y^{n-k}, \label{thm:qbinomial2}
\end{align}
which is the $n,k \geq 0$ case of Theorem \ref{thm:qbinomial}.
In \cite{Schl1}, the second author generalized the noncommutative $q$-binomial
theorem \eqref{thm:qbinomial2} to a weight-dependent binomial theorem
for \defn{weight-dependent binomial coefficients} (see
Theorem~\ref{thm:binomial} below) and gave a combinatorial interpretation of
these coefficients in terms of lattice paths. Specializing the general weights
of the weight-dependent binomial coefficients, one obtains some interesting
special cases, such as elliptic binomial coefficients and symmetric functions
as well as the $q$- and ordinary binomial coefficients. In the conclusion of
their paper \cite{FS}, Formichella and Straub ask whether these elliptic
binomial coefficients also permit a natural extension to negative numbers.
We will show that such an extension is possible even for the more general
weight-dependent binomial coefficients.

In this paper, we extend the weight-dependent binomial coefficients to
negative integer values, analogous to the work of Loeb \cite{L} and
Formichella and Straub \cite{FS}. Remarkably, many of their results continue
to hold in the general weighted setting. We study reflection formulae inspired
by an involution which recently appeared in \cite{SY}. We give a combinatorial
interpretation of the weight-dependent binomial coefficients in terms of
lattice paths and prove a weight-dependent generalization of
Theorem~\ref{thm:qbinomial}. As a corollary of the binomial theorem, we obtain
convolution formulae analogous to the Chu--Vandermonde convolution formula,
which we also are able to interpret combinatorially. Our combinatorial
interpretation of the weight-dependent binomial coefficients in terms of
lattice paths translates, via a one-to-one correspondence, to a combinatorial
interpretation in terms of hybrid sets (as studied by Loeb \cite{L} in the
classical case and by Formichella and Straub \cite {FS} in the $q$-case)
which we make explicit. Finally, we study some important special cases of the
weight-dependent binomial coefficients, such as elementary and complete
homogeneous symmetric functions (with application of these cases to Stirling
numbers), and elliptic binomial coefficients. 

\section{Weight-dependent commutation relations}
\subsection{A noncommutative algebra}
Let $(w(s,t))_{s,t\in\mathbb{Z}}$ be a sequence of invertible variables.
We start by extending the algebra $\C_q[x,y]$ to the weight-dependent
setting with invertible variables $x$ and $y$.

\begin{definition}\label{def:Cwxy3}
For a doubly-indexed sequence of invertible variables
$(w(s,t))_{s,t\in\mathbb{Z}}$, let $\mathbb{C}_w [x,x^{-1},y,y^{-1}]$ be the
associative unital algebra over $\mathbb{C}$ generated by $x$, $x^{-1}$, $y$
and $y^{-1}$ and the sequence of invertible variables
$(w(s,t)^{\pm 1})_{s,t\in \Z}$ satisfying the following relations:
	\begin{subequations}\label{eqn:noncommrel3}
		\begin{align}
			x^{-1}x&=xx^{-1}=1\label{subeqn:rel1}\\
			y^{-1}y&=yy^{-1}=1\label{subeqn:rel2}\\
			yx &= w(1,1)xy,\label{subeqn:rel3}\\
			x w(s,t)&= w(s+1,t)x,\label{subeqn:rel4}\\
			y w(s,t)&= w(s,t+1)y\label{subeqn:rel5},
		\end{align}
	\end{subequations}
	for all $s,t\in\mathbb Z$. 
\end{definition}
\noindent Note that from the above relations also the following relations
can be obtained: 
	\begin{subequations}\label{eqn:noncommrelbonus}
	\begin{align}
		x^{-1}y&=w(0,1)yx^{-1},\tag{2.1f}\\
		xy^{-1}&=w(1,0)y^{-1}x,\tag{2.1g}\\
		y^{-1}x^{-1}&=w(0,0)x^{-1}y^{-1},\tag{2.1h} \\
		x^{-1} w(s,t)&= w(s-1,t)x^{-1},\tag{2.1i}\\
		y^{-1} w(s,t)&= w(s,t-1)y^{-1}\tag{2.1j}.
	\end{align}
\end{subequations}
In this paper, we will frequently restrict the just defined algebra
$\mathbb{C}_w [x,x^{-1},y,y^{-1}]$ to $\mathbb{C}_w [x,y,y^{-1}]$
(where we allow no negative powers of $x$ and omit the relation
\eqref{subeqn:rel1}) or to $\mathbb{C}_w [x,x^{-1},y]$ (where we allow
no negative powers of $y$ and omit the relation \eqref{subeqn:rel2}),
respectively. In addition we find it convenient to work in the extensions
of these algebras to the algebras of formal power series
$\mathbb{C}_w [[x,x^{-1},y,y^{-1}]]$, $\mathbb{C}_w [[x,y,y^{-1}]]$ and
$\mathbb{C}_w [[x,x^{-1},y]]$ (keeping the same relations). Note that
expressions of the form $(x+y)^n$ with $n<0$ do not have a unique power
series expansion in $\mathbb{C}_w [[x,x^{-1},y,y^{-1}]]$. For $n<0$ we have
to decide how to expand $(x+y)^n$, say as $(x+y)^n=((1+xy^{-1})y)^n$ as an
element in $\mathbb{C}_w [[x,y,y^{-1}]]$, or rather as
$(x+y)^n=(x(1+x^{-1}y))^n$ as an element in $\mathbb{C}_w [[x,x^{-1},y]]$,
respectively. By default, we shall always choose the algebra of formal power
series $\mathbb{C}_w [[x,y,y^{-1}]]$ (allowing no negative powers of $x$)
and only resort to $\mathbb{C}_w [[x,x^{-1},y]]$ (allowing no negative powers
of $y$) if we explicitly mention that.

For $l,m \in \Z \cup \{\pm \infty \}$ we define products of
(possibly noncommutative) invertible variables $A_j$ as follows:
\begin{equation}\label{prod}
	\prod_{j=l}^mA_j=\left\{
	\begin{array}{ll}
		A_lA_{l+1}\dots A_m&m>l-1\\
		1&m=l-1\\
		A_{l-1}^{-1}A_{l-2}^{-1}\dots A_{m+1}^{-1}&m<l-1
	\end{array}\right..
\end{equation}
Note that
\begin{equation}\label{invprod}
	\prod_{j=l}^mA_j=\prod_{j=m+1}^{l-1}A_{m+l-j}^{-1},
\end{equation}
for all $l,m\in\Z\cup\{\pm\infty\}$.

Especially for the reflection formulae it will be necessary to define
$\mathrm{sgn}(n)$ for $n\in \Z$, following \cite{FS,Spr}, as 
\begin{equation}\label{def:sgn}
	\mathrm{sgn}(n)=
	\begin{cases} 1 & n \geq 0 \\ -1 & n < 0. \end{cases}
\end{equation}

For $s,t\in \Z$ and the sequence of invertible weights $(w(s,t))_{s,t\in\Z}$
we define 
\begin{equation}\label{eq:W}
	W(s,t):= \prod_{j=1}^t w(s,j).
\end{equation}

Note that for $s,t\in\mathbb{Z}$, we have $w(s,t)=W(s,t)/W(s,t-1)$.
We refer to the $w(s,t)$ as \defn{small weights}, whereas to the
$W(s,t)$ as \defn{big weights}.

\begin{lemma}\label{lem:inv}
Let $(w(s,t))_{s,t\in\mathbb Z}$ be a doubly-indexed sequence of invertible
variables, and $x$ and $y$ two additional invertible variables, together
forming the associative algebra
$A_{x,y}=\mathbb{C}_{w_{x,y}} [x,x^{-1},y,y^{-1}]$ where $w_{x,y}(s,t)=w(s,t)$.
Then the following six homomorphisms are involutive algebra isomorphisms.
\begin{subequations}\label{eqn:inv}
\begin{align}
  &\phi_{y,x}:& A_{x,y}&\to A_{y,x}& \; \text{with} \;&& w_{y,x}(s,t)
  &=w(t,s)^{-1}, \label{eqn:inv1}	\\
  &\phi_{x^{-1},y}:& A_{x,y}&\to A_{x^{-1},y}& \; \text{with} \;
  && w_{x^{-1},y}(s,t) &=w(1-s,t)^{-1}, \label{eqn:inv2} \\
  &\phi_{x^{-1},x^{-1}y}:& A_{x,y}&\to A_{x^{-1},x^{-1}y}
  & \; \text{with} \;&& w_{x^{-1},x^{-1}y}(s,t)
  &=w(1-s-t,t)^{-1}, \label{eqn:inv3} \\	
  &\phi_{x,y^{-1}}:& A_{x,y}&\to A_{x,y^{-1}}& \; \text{with} \;
  && w_{x,y^{-1}}(s,t)&=w(s,1-t)^{-1}, \label{eqn:inv4} \\	
  &\phi_{x^{-1},y^{-1}}:& A_{x,y}&\to A_{x^{-1},y^{-1}}& \; \text{with} \;
  && w_{x^{-1},y^{-1}}(s,t)&=w(1-s,1-t), \label{eqn:inv5} \\	
  &\phi_{y^{-1}x,y^{-1}}:& A_{x,y}&\to A_{y^{-1}x,y^{-1}}& \; \text{with} \;
  && w_{y^{-1}x,y^{-1}}(s,t)&=w(s,1-s-t)^{-1}. \label{eqn:inv6}
		\end{align}
	\end{subequations}
\end{lemma}
It is straightforward to check that the simultaneous replacement of
$w_{x,y}(s,t)$ ($s,t\in\mathbb Z$), $x$ and $y$ in
\eqref{eqn:inv1}--\eqref{eqn:inv6} by $w_{x',y'}(s,t)$, $x'$ and $y'$,
respectively, again satisfies the conditions in \eqref{eqn:noncommrel3}.
Note that the involutions \eqref{eqn:inv4}--\eqref{eqn:inv6} can be realized
as compositions of \eqref{eqn:inv1}--\eqref{eqn:inv3}. For example,
$\phi_{y^{-1}x,y^{-1}} = \phi_{y,x} \circ \phi_{x^{-1},x^{-1}y} \circ \phi_{y,x}$.
Further note that Lemma~\ref{lem:inv} also holds for the algebras
$\C_w[x,x^{-1},y]$ and $\C_w[x,y,y^{-1}]$ (when the respective homomorphisms
can be defined) as well as for their respective formal power series extensions. 

Lemma~\ref{lem:inv} is an extension of \cite[Lemma 2]{SY}. There, the
involution $\phi_{x^{-1},x^{-1}y}$ with corresponding weight function
$\widetilde{w}(s,t)\eqdef w_{x^{-1},x^{-1}y}(s,t)=w(1-s-t,t)^{-1}$ was used in
the algebra $\C_w[x,x^{-1},y]$ to construct weight-dependent Fibonacci
polynomials satisfying a noncommutative weight-dependent Euler--Cassini
identity.

It is clear that, given an identity in the variables $w(s,t)$
($s,t\in\mathbb Z$), $x$ and $y$, subject to the commutation relations
\eqref{def:Cwxy3}, a new valid identity can be obtained by applying the
isomorphism $\phi$ to each side of the identity where the variables still
satisfy the commutation relations \eqref{def:Cwxy3}. We will apply such
isomorphisms in the proofs of Lemma~\ref{lem:com} and
Theorem~\ref{thm:extbinomial}.

\begin{remark}
While the six isomorphisms in Lemma~\ref{lem:inv} are involutions, i.e.,
are of order $2$, it is also possible to specify isomorphisms of order $3$.
In particular, for the homomorphisms
\begin{subequations}\label{eqn:3inv}
\begin{align}
  &\phi_{y^{-1},xy^{-1}}:& A_{x,y}&\to A_{y^{-1},xy^{-1}}& \; \text{with} \;
  && w_{y^{-1},xy^{-1}}(s,t)&=w(t,2-s-t), \label{eqn:inv7} \\
  &\phi_{yx^{-1},x^{-1}}:& A_{x,y}&\to A_{yx^{-1},x^{-1}}& \; \text{with} \;
  && w_{yx^{-1},x^{-1}}(s,t)&=w(2-s-t,s), \label{eqn:inv8} \\
  &\phi_{y,x^{-1}y^{-1}}:& A_{x,y}&\to A_{y,x^{-1}y^{-1}}& \; \text{with} \;
  && w_{y,x^{-1}y^{-1}}(s,t)&=w(1-t,1+s-t), \label{eqn:inv9} \\
  &\phi_{x^{-1}y^{-1},x}:& A_{x,y}&\to A_{x^{-1}y^{-1},x}& \; \text{with} \;
  && w_{x^{-1}y^{-1},x}(s,t)&=w(t-s,1-s), \label{eqn:inv10}
\end{align}
\end{subequations}
we have $\phi_{y^{-1},xy^{-1}}^3=\phi_{yx^{-1},x^{-1}}^3=\phi_{y,x^{-1}y^{-1}}^3
=\phi_{x^{-1}y^{-1},x}^3=id_{A_{x,y}}$, moreover,
$\phi_{y^{-1},xy^{-1}}^{-1}=\phi_{yx^{-1},x^{-1}}$ and
$\phi_{y,x^{-1}y^{-1}}^{-1}=\phi_{x^{-1}y^{-1},x}$.
We will not make use of these isomorphisms of order $3$ in this paper
(nor do we provide combinatorial interpretations in terms of lattice paths
or hybrid sets here).
\end{remark}

The following rule for interchanging powers of $x$ and $y$ is an extension
to integer values of a corresponding lemma in \cite[Lemma 1]{Schl1}:
\begin{lemma}\label{lem:com}
	For all $k,\ell \in \Z$ we have
	$$
	y^k x^\ell=\left(\prod_{i=1}^\ell\prod_{j=1}^k w(i,j)\right)x^\ell y^k
	=\left(\prod_{i=1}^\ell W(i,k)\right)x^\ell y^k.
	$$
\end{lemma} 
\begin{proof}
The case $k,\ell \geq 0$ is already given in \cite{Schl1} and is easy to
prove by induction; we therefore omit the proof. 
For $k\geq 0$ and $\ell < 0$ we combine the involution \eqref{eqn:inv2}
with the $k,\ell \geq 0$ case and use identity \eqref{invprod} to obtain:
$$y^k x^\ell= y^k (x^{-1})^{-\ell}
= \left(\prod_{i=1}^{-\ell} \prod_{j=1}^k w(1-i,j)^{-1} \right)
(x^{-1})^{-\ell}y^k
= \left(\prod_{i=1}^{\ell} \prod_{j=1}^k w(i,j) \right)x^\ell y^k .$$
The remaining cases can be proved in the same manner by applying the
involutions \eqref{eqn:inv4} and \eqref{eqn:inv5} to the $k,\ell\geq 0$ case.
\end{proof}

\subsection{Weight-dependent binomial coefficients with integer values}
Now we are ready to define weight-dependent binomial coefficients for all
integer values. Let the \defn{weight-dependent binomial coefficients}
or \defn{$w$-binomial coefficients} be defined by 
\begin{subequations}\label{def:wbineq}
\begin{align}
&\wBinom{n}{0}{w}=\wBinom{n}{n}{w} =1
\qquad\text{for\/ $n\in \mathbb{Z}$ },\label{def:wbineq1}\\
\intertext{and for $n,k\in\mathbb{Z}$, provided that $(n+1,k)\neq(0,0)$,}
&\wBinom{n+1}{k}{w}=
\wBinom{n}{k}{w}
+\wBinom{n}{k-1}{w}
\,W(k,n+1-k). \label{def:wbineq2}
\end{align}
\end{subequations}
\begin{example}[$n=-1$]\label{ex:minus1}
Using induction separately for $k\geq 0$ and for $k<0$ we obtain that
the weight-dependent binomial coefficient for $n=-1$ evaluates to
$$
\wBinom{-1}{k}{w}=(-1)^{k} \mathrm{sgn}(k) \prod_{j=1}^{k} W(j,-j),
$$
where $\mathrm{sgn}(k)$ is defined in \eqref{def:sgn}.
\end{example}
For $n,k \geq 0$ in \eqref{def:wbineq}, the $w$-binomial coefficients
coincide with the weight-dependent binomial coefficients in \cite{Schl1},
which have a combinatorial interpretation in terms of
\emph{weighted lattice paths} (see also \cite{Schl0} for the elliptic case).
Here, for $n,m\geq 0$, a lattice path is a sequence of north and east steps
in the first quadrant of the $xy$-plane, starting at the origin $(0,0)$ and
ending at $(n,m)$.
We give weights to such paths by assigning the big weight $W(s,t)$
to each east step $(s-1,t)\rightarrow (s,t)$ and $1$ to each north step.
Then define the weight of a path $P$, $w(P)$, to be the product of the
weights of all its steps. An example is given in Figure~\ref{fig:path1}.

Given two points $A,B\in\naturals_0^2$, let $\mathcal{P}(A\rightarrow B)$
be the set of all lattice paths from $A$ to $B$, and define 
\begin{displaymath}
w(\mathcal{P}(A\rightarrow B)):= \sum_{P\in \mathcal{P}(A\rightarrow B)}w(P).
\end{displaymath}
Then we have 
\begin{equation}
w(\mathcal{P}((0,0)\rightarrow (k,n-k)))=\wBinom{n}{k}{w}\label{eqn:wbin}
\end{equation}
as both sides of the equation satisfy the same recursion
and initial conditions.

Interpreting the $x$-variable as an east step and the $y$-variable
as a north step, we get the following weight-dependent binomial theorem.
\begin{theorem}[\cite{Schl1}]\label{thm:binomial}
	Let $n\in\mathbb{N}_0$. Then, as an identity in $\mathbb{C}_w[x,y]$,
	\begin{equation}\label{eqn:binomial}
		(x+y)^n =\sum_{k=0}^n
		\wBinom{n}{k}{w} x^k y^{n-k}.
	\end{equation}
\end{theorem}
Our goal is to extend this weight-dependent binomial theorem to arbitrary
integer exponents and therefore generalize Theorems \ref{thm:loeb_binom} and
\ref{thm:qbinomial}.

Before stating the general binomial theorem, let us extend the lattice path
model of \cite{Schl0} to all $n,k \in \Z$. Let a
\emph{weighted hybrid lattice path} be a sequence of steps in the $xy$-plane
starting at the origin $(0,0)$ and ending at $(n,m)$ with $n,m \in \Z$ using
the following steps:
\begin{enumerate}
\item If $n,m \geq 0$, we use north and east steps 
(
\begin{tikzpicture}[scale=0.3]
	\draw[-to] (0,0) -- (0,1); 
\end{tikzpicture} 
,
\begin{tikzpicture}[scale=0.3]
\draw[-to] (0,.5) -- (1,.5);
\draw[color=white,line width=0.01pt] (-0.1,0) -- (-0.1,1);
\end{tikzpicture} 
	)
	,
      \item if $n\geq 0$ and $m<0$,
        we use south steps and east-south step combinations 
(
\begin{tikzpicture}[scale=0.3]
\draw[-to] (0,0) -- (0,-1); 
\end{tikzpicture} 
,
\begin{tikzpicture}[scale=0.3]
\draw[-to] (0,0) -- (1,0) -- (1,-1); 
\end{tikzpicture} 
	)
and every path starts with a south step,
\item if $n<0$ and $m\geq 0$,
  we use north-west step combinations and west steps 
(
\begin{tikzpicture}[scale=0.3]
\draw[-to] (0,0) -- (0,1) -- (-1,1); 
\end{tikzpicture} 
,
\begin{tikzpicture}[scale=0.3]
\draw[white] (-0.5,0) -- (-0.5,1);
\draw[-to] (0,0.5) -- (-1,0.5); 
\end{tikzpicture}
)
and every path starts with a west step,
\item if $n,m<0$, there are no allowed steps.
\end{enumerate}
Figure~\ref{fig:steps2} shows the possible steps of a hybrid lattice path.
The arrows indicate the direction of the steps. Note that in the region
$m<0\leq n$ every east step has to be followed by a south step and in the
region $n<0\leq m$ every north step has to be followed by a west step
(which is not evident in the figure but follows from the step combinations
in the definition). The figure also shows that some points are not reachable
using these steps. This corresponds to the fact that the $w$-binomial
coefficient with integer values is $0$ in some regions which we will specify
in Equation~\eqref{eqn:zeroregions}.
\begin{center}
	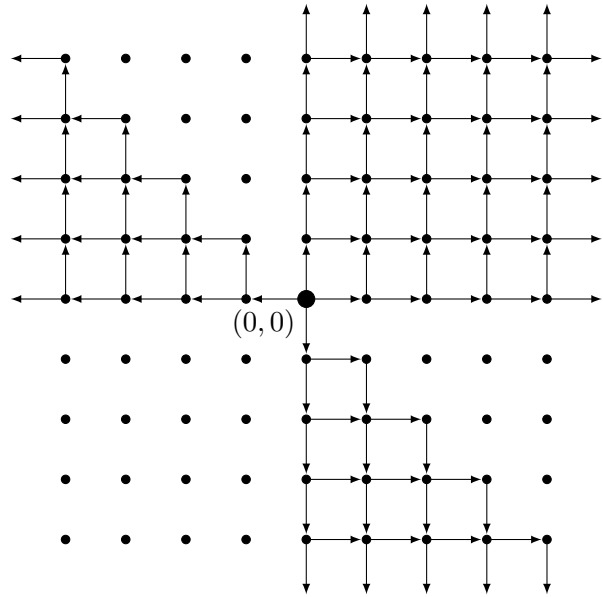
\begin{figure}
		\begin{tikzpicture}[scale=0.8]
			\foreach \x in {-4,...,4}
			\foreach \y in {-4,...,4}
			\draw[fill=black](\x,\y)circle(2pt);
			\foreach \x in {0,...,4}
			\foreach \y in {0,...,4}
			{\draw[->,>=latex,shorten >=2pt](\x,\y)--(\x+1,\y);
				\draw[->,>=latex,shorten >=2pt](\x,\y)--(\x,\y+1);}
			\foreach \x in {0,...,4}
			\foreach \y in {\x,...,4}
			\draw[->,>=latex,shorten >=2pt](\x,-\y)--(\x,-\y-1);
			\foreach \x in {0,...,3}
			\foreach \y in {\x,...,3}
			\draw[->,>=latex,shorten >=2pt](\x,-\y-1)--(\x+1,-\y-1);
			\foreach \x in {0,...,3}
			\foreach \y in {0,...,\x}
			\draw[->,>=latex,shorten >=2pt](-\x-1,\y)--(-\x-1,\y+1);
			\foreach \x in {0,...,4}
			\foreach \y in {0,...,\x}
				\draw[->,>=latex,shorten >=2pt](-\x,\y)--(-\x-1,\y);
			\draw[fill=black](0,0)circle(4pt);
			\node[below left] at (0,0) {$(0,0)$};
\end{tikzpicture}
		\caption{The possible steps of a hybrid lattice path.}
		\label{fig:steps2}
	\end{figure}
\end{center}
\noindent We give weights to such paths by assigning the weights
\begin{itemize}
	\item $1$ to each regular north step and south step,
	\item $(-1)$ to each north and south step in a north-west
          or an east-south step combination,
	\item $W(s,t)$ to each east step $(s-1,t) \rightarrow (s,t)$, and
	\item $W(s,t)^{-1}$ to each west step $(s-1,t) \leftarrow (s,t)$.
\end{itemize}
To illustrate it graphically, we assign the weights
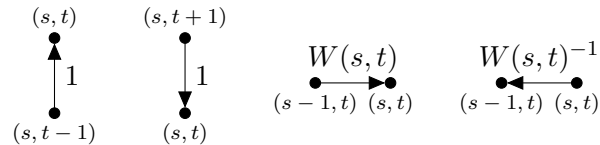
\begin{center}
	\begin{tabular}{cccc}
		\begin{tikzpicture}
			\draw[fill=black](0,0)circle(2pt) node[below] {\scriptsize $(s,t-1)$};
			\draw[fill=black](0,1)circle(2pt) node[above] {\scriptsize $(s,t)$};
			\draw[->,>=triangle 45, shorten >=2pt] (0,0)--(0,1) ;
			\node[right] at (0,0.5) {$1$};
		\end{tikzpicture}&
		\begin{tikzpicture}
			\draw[fill=black](0,0)circle(2pt) node[below] {\scriptsize $(s,t)$};
			\draw[fill=black](0,1)circle(2pt) node[above] {\scriptsize $(s,t+1)$};
			\draw[->,>=triangle 45, shorten >=2pt] (0,1)--(0,0) ;
			\node[right] at (0,0.5) {$1$};
		\end{tikzpicture}&
		\begin{tikzpicture}	
			\node[color=white] at (0.5,-0.3) {.};
			\draw[fill=black](0,0.5)circle(2pt) node[below] {\scriptsize $(s-1,t)$};
			\draw[fill=black](1,0.5)circle(2pt) node[below] {\scriptsize $(s,t)$};
			\draw[->,>=triangle 45, shorten >=2pt] (0,0.5)--(1,0.5) ;
			\node[above] at (0.5,0.5) {$W(s,t)$};
		\end{tikzpicture}&
		\begin{tikzpicture}	
			\node[color=white] at (0.5,-0.3) {.};
			\draw[fill=black](0,0.5)circle(2pt) node[below] {\scriptsize $(s-1,t)$};
			\draw[fill=black](1,0.5)circle(2pt) node[below] {\scriptsize $(s,t)$};
			\draw[->,>=triangle 45,shorten >=2pt] (1,0.5)--(0,0.5) ;
			\node[above] at (0.5,0.5) {$W(s,t)^{-1}$};
		\end{tikzpicture}
	\end{tabular}
\end{center}
and, if we highlight the step combinations by rounded corners, 
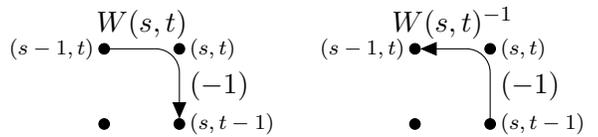
\begin{center}
	\begin{tabular}{cc}	
		\begin{tikzpicture}	
			\draw[fill=black](0,0)circle(2pt);
			\draw[fill=black](1,0)circle(2pt) node[right] {\scriptsize $(s,t-1)$};
			\draw[fill=black](0,1)circle(2pt) node[left] {\scriptsize $(s-1,t)$};
			\draw[fill=black](1,1)circle(2pt) node[right] {\scriptsize $(s,t)$};
			\draw[->,>=triangle 45,rounded corners=9pt,shorten >=2pt] (0,1)--(1,1)--(1,0) ;
			\node[above] at (0.5,1) {$W(s,t)$};
			\node[right] at (1,0.5) {$(-1)$};
		\end{tikzpicture}&
		\begin{tikzpicture}
			\draw[fill=black](0,0)circle(2pt);
			\draw[fill=black](1,0)circle(2pt) node[right] {\scriptsize $(s,t-1)$};
			\draw[fill=black](0,1)circle(2pt) node[left] {\scriptsize $(s-1,t)$};
			\draw[fill=black](1,1)circle(2pt) node[right] {\scriptsize $(s,t)$};
			\draw[->,>=triangle 45,rounded corners=9pt,shorten >=2pt] (1,0)--(1,1)--(0,1) ;
			\node[above] at (0.5,1) {$W(s,t)^{-1}$};
			\node[right] at (1,0.5) {$(-1)$};
		\end{tikzpicture}
	.
	\end{tabular}
\end{center}

\noindent
The weight of a path $P$, $w(P)$, is again defined as the product of the
weights of all its steps. 
\begin{example}
Figure~\ref{fig:path1} shows a hybrid lattice path in the area $n,m \geq 0$
with weight $W(1,0) \cdot 1 \cdot W(2,1) \cdot 1 \cdot W(3,2)\cdot W(4,2)
= w(2,1)w(3,1)w(3,2)w(4,1)w(4,2)$. Paths in this area correspond to
(ordinary) weighted lattice paths.
\begin{figure}
		\begin{tikzpicture}[scale=1.4]
	\path[fill=gray!20] (1,0) rectangle (2,1) node[pos=.5, color=black] {\scriptsize $w(2,1)$};
	\path[fill=gray!20] (2,0) rectangle (3,1) node[pos=.5, color=black] {\scriptsize $w(3,1)$};
	\path[fill=gray!20] (3,0) rectangle (4,1) node[pos=.5, color=black] {\scriptsize $w(4,1)$};
	\path[fill=gray!20] (2,1) rectangle (3,2) node[pos=.5, color=black] {\scriptsize $w(3,2)$};
	\path[fill=gray!20] (3,1) rectangle (4,2) node[pos=.5, color=black] {\scriptsize $w(4,2)$};
	\draw[->,>=latex] (0,-0.5) -- (0,2.5);
	\draw[->,>=latex] (-0.5,0) -- (4.5,0);
	\foreach \x in {0,...,4}
	\draw[dotted] (\x,0)--(\x,2);
	\foreach \y in {0,...,2}
	\draw[dotted] (0,\y)--(4,\y);
	\draw[line width=2pt,->,>=latex, shorten >=2pt,myblue] (0,0) -- (1,0) -- (1,1)--(2,1) -- (2,2)--(3,2)--(4,2);
	\foreach \x in {0,...,4}
	\foreach \y in {0,...,2}
	\draw[fill=black](\x,\y)circle(2pt);
	\node[above,myblue] at (0.5,0) {\tiny $W(1,0)$};	
	\node[above,myblue] at (1.5,1) {\tiny $W(2,1)$};	
	\node[above,myblue] at (2.5,2) {\tiny $W(3,2)$};	
	\node[above,myblue] at (3.5,2) {\tiny $W(4,2)$};		
	\node[left,myblue] at (1,0.6) {\scriptsize $1$};
	\node[left,myblue] at (2,1.6) {\scriptsize $1$};	
	\node[above left] at (0,0) {$(0,0)$};
\end{tikzpicture}
		\caption{A hybrid lattice path in the area $n,m \geq 0$.}
		\label{fig:path1}
\end{figure}
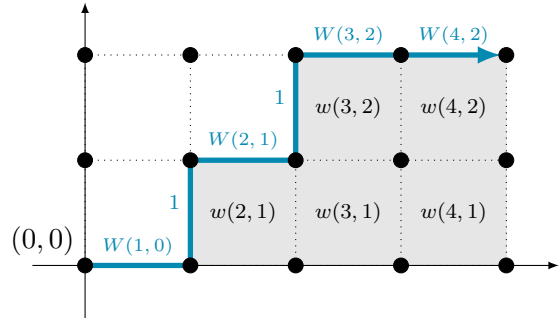
The left side of Figure~\ref{fig:path2} shows a hybrid lattice path in the
area $m<0\leq n$ with weight
$1\cdot W(1,-1) \cdot (-1) \cdot 1 \cdot W(2,-3) \cdot (-1)
= (-1)^2 w(1,0)^{-1} w(2,0)^{-1} w(2,-1)^{-1} w(2,-2)^{-1}$. The right side
of Figure~\ref{fig:path2} shows a hybrid lattice path in the area
$n<0\leq m$ with weight $W(0,0)^{-1} \cdot (-1) \cdot W(-1,1)^{-1}
\cdot W(-2,1)^{-1} \cdot (-1) \cdot W(-3,2)^{-1}
= (-1)^{2} w(-1,1)^{-1}w(-2,1)^{-1}w(-3,1)^{-1}w(-3,2)^{-1}$. 
\begin{center}
\begin{figure}
\begin{subfigure}{0.4\textwidth}
			\begin{tikzpicture}[scale=1.5]
				\node[color=white] at (-1,0) {$.$};
				\node[color=white] at (3,0) {$.$}; 
				\path[fill=gray!20] (0,0) rectangle (1,-1) node[pos=.5, color=black] {\scriptsize $w(1,0)^{-1}$};
				\path[fill=gray!20] (1,0) rectangle (2,-1) node[pos=.5, color=black] {\scriptsize $w(2,0)^{-1}$};
				\path[fill=gray!20] (1,-1) rectangle (2,-2) node[pos=.5, color=black] {\scriptsize $w(2,-1)^{-1}$};
				\path[fill=gray!20] (1,-2) rectangle (2,-3) node[pos=.5, color=black] {\scriptsize $w(2,-2)^{-1}$};
				\draw[->,>=latex] (0,-4.3) -- (0,0.3);
				\draw[->,>=latex] (-0.3,0) -- (2.3,0);
				\foreach \x in {0,...,2}
					\draw[dotted] (\x,0)--(\x,-4);
				\foreach \y in {-4,...,0}
					\draw[dotted] (0,\y)--(2,\y);
				\draw[line width=2pt, myblue] (0,0)--(0,-1);
				\draw[line width=2pt,rounded corners=9pt, myblue] (0,-1) -- (1,-1)--(1,-3);
				\draw[line width=2pt,->,>=latex, shorten >=2pt,rounded corners=9pt,myblue] (1,-3) -- (2,-3) --(2,-4);
				\foreach \x in {0,...,2}
					\foreach \y in {-4,...,0}
						\draw[fill=black](\x,\y)circle(2pt);
				\node[below,myblue] at (0.5,-1) {\scriptsize $W\!(1,-1)$};	
				\node[below,myblue] at (1.5,-3) {\scriptsize $W\!(2,-3)$};	
				\node[left,myblue] at (0,-0.5) {\scriptsize $1$};
				\node[left,myblue] at (1,-2.5) {\scriptsize $1$};	
				\node[left,myblue] at (1,-1.6) {\scriptsize $(-1)$};
				\node[left,myblue] at (2,-3.6) {\scriptsize $(-1)$};
				\node[above left] at (0,0) {$(0,0)$};
			\end{tikzpicture}
		\end{subfigure}
		\begin{subfigure}{0.5\textwidth}
			\begin{tikzpicture}[scale=1.5]
				\draw[->,>=latex] (0,-0.3) -- (0,2.3);
				\draw[->,>=latex] (-4.3,0) -- (0.3,0);
				\path[fill=gray!20] (-2,0) rectangle (-1,1) node[pos=.5, color=black] {\scriptsize $w(-1,1)^{-1}$};
				\path[fill=gray!20] (-3,0) rectangle (-2,1) node[pos=.5, color=black] {\scriptsize $w(-2,1)^{-1}$};
				\path[fill=gray!20] (-4,0) rectangle (-3,1) node[pos=.5, color=black] {\scriptsize $w(-3,1)^{-1}$};
				\path[fill=gray!20] (-4,1) rectangle (-3,2) node[pos=.5, color=black] {\scriptsize $w(-3,2)^{-1}$};
				\foreach \x in {-4,...,0}
					\draw[dotted] (\x,0)--(\x,2);
				\foreach \y in {0,...,2}
					\draw[dotted] (0,\y)--(-4,\y);
				\draw[line width=2pt,myblue] (0,0)--(-1,0);
				\draw[line width=2pt, rounded corners=9pt, myblue] (-1,0) -- (-1,1) -- (-3,1);
				\draw[line width=2pt,->,>=latex, shorten >=2pt,rounded corners=9pt,myblue] (-3,1) -- (-3,2) -- (-4,2);
				\foreach \x in {-4,...,0}
					\foreach \y in {0,...,2}
						\draw[fill=black](\x,\y)circle(2pt);				
				\node[above,myblue] at (-0.5,0) {\scriptsize $W\!(0,0)^{-1}$};	
				\node[above,myblue] at (-1.5,1) {\scriptsize $W\!(-1,1)^{-1}$};
				\node[above,myblue] at (-2.5,1) {\scriptsize $W\!(-2,1)^{-1}$};
				\node[above,myblue] at (-3.5,2) {\scriptsize $W\!(-3,2)^{-1}$};
				\node[right,myblue] at (-1,0.6) {\scriptsize $(-1)$};
				\node[right,myblue] at (-3,1.6) {\scriptsize $(-1)$};
				\node[below right] at (0,0) {$(0,0)$};	
			\end{tikzpicture}
		\end{subfigure}
 \caption{A hybrid lattice path in the area $m<0\leq n$ (left) and a path
   in the area $n<0\leq m$ (right). The (diagonal) step combinations
   are indicated by rounded corners. The negative inner corners
   (see Remark \ref{rem:area}) are marked with red diagonal lines.}
	\label{fig:path2}
\end{figure}
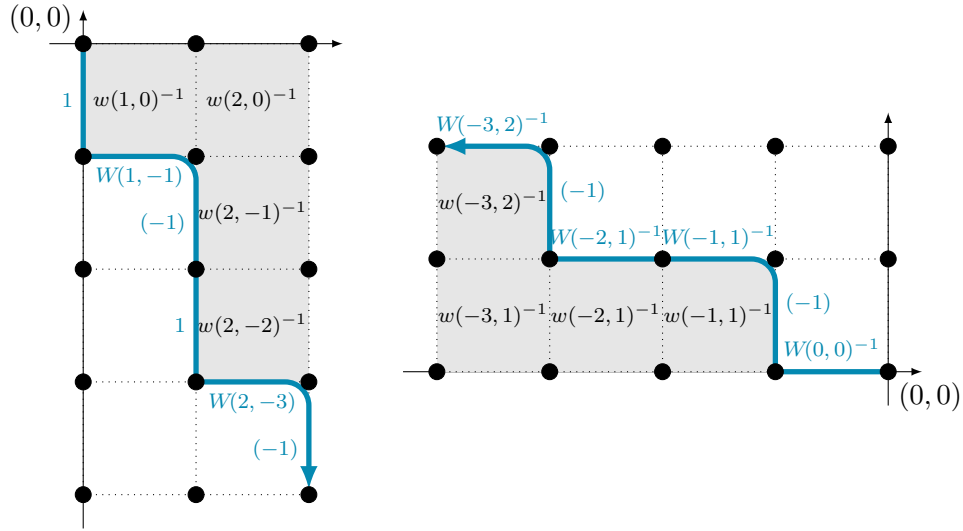
\end{center}
\end{example}
Given two points $A,B\in\Z^2$,
let $\mathcal{P}_{\Z}(A\rightarrow B)$ be the set of all hybrid lattice
paths from $A$ to $B$, and define 
\begin{displaymath}
  w(\mathcal{P}_{\Z}(A\rightarrow B)):=
  \sum_{P\in \mathcal{P}_{\Z}(A\rightarrow B)}w(P).
\end{displaymath}

\begin{theorem}\label{thm:hlp}
	Let $n,k\in\Z$. Then,
\begin{equation}
  w(\mathcal{P_{\Z}}((0,0)\rightarrow (k,n-k)))
  = \wBinom{n}{k}{w}\label{eqn:wbinext}.
\end{equation}
\end{theorem}
\begin{proof}
If $0 \leq k \leq n$ we are left with Equation \eqref{eqn:wbin}. Let
$0 \leq n < k$, $k<0\leq n$ or $n<k<0$, then there are no possible hybrid
lattice paths and indeed, it is not hard to use induction to prove that 
\begin{align}\label{eqn:zeroregions}
	&\wBinom{n}{k}{w}=0
	\qquad\text{for\/ $0 \leq n < k$, $k<0\leq n$ or $n<k<0$}.
	\end{align}
The remaining cases $n<0\leq k$ and $k\leq n <0$ can be proved using induction.

For $n<0\leq k$ we are dealing with the set of lattice paths from $(0,0)$
to $(k,n-k)$ using south steps and east-south step combinations. We begin
with the initial conditions
\begin{equation*}
  \wBinom{n}{0}{w}=1  \;\text{ and }\; \wBinom{-1}{k}{w}
                       =(-1)^{k} \prod_{j=1}^{k} W(j,-j)
\end{equation*} 
for $n<0$ and $k\geq0$ and indeed, there is only one path from $(0,0)$
to $(0,n)$ consisting only of south steps with weight $1$ and one path
from $(0,0)$ to $(k,-1-k)$ consisting only of east-south step combinations
with weight $(-1)^{k} \prod_{j=1}^{k} W(j,-j)$.

Now assume the result holds when $n+1\leq-1$ or $k-1\geq0$. The last step
of the hybrid lattice path ending at $(k,n-k)$ is either a south step
starting at $(k,n-k+1)$ or an east-south step combination starting at
$(k-1,n-k+1)$. We obtain from \eqref{def:wbineq2} that
\begin{equation}
\wBinom{n}{k}{w} = \wBinom{n+1}{k}{w} - \wBinom{n}{k-1}{w} W(k,n+1-k).
\end{equation}
By induction, the first term of the right hand side is the weighted counting
of hybrid lattice paths ending at $(k,n-k+1)$ combined with a south step with
weight $1$. The second term is the weighted counting of paths ending at
$(k-1,n-k+1)$ combined with an east-south step combination with weight
$(-1)\cdot W(k,n-k+1)$. This completes the proof for the case $n<0\leq k$.

The case $k\leq n <0$ can be proved analogously using the following third
form of the recurrence relation \eqref{def:wbineq2}:
\begin{equation}
  \wBinom{n}{k}{w} = \wBinom{n+1}{k+1}{w} W(k+1,n-k)^{-1}
  - \wBinom{n}{k+1}{w} W(k+1,n-k)^{-1},
\end{equation}
where the first term of the sum corresponds to the weighted counting of
hybrid lattice paths ending at $(k+1,n-k)$ combined with a west step with
weight $W(k+1,n-k)^{-1}$ and the second term to the weighted counting of
paths ending at $(k+1,n-k-1)$ combined with a north-west step combination
with weight $(-1)\cdot W(k+1,n-k)^{-1}$. 
\end{proof}

\begin{remark}\label{rem:area}
  As the Figures~\ref{fig:path1}~and~\ref{fig:path2} show, we can interpret
  the weight of the path as a weight function corresponding to the area
  between the path and the $x$-axis. For $s,t \in \Z$, let the cell $(s,t)$
  be the unit square with north-east corner at $x=s$ and $y=t$. For a hybrid
  lattice path $P$ we define the set $\mathcal{C}(P)$ as the collection of
  cells $(s,t)$ between the path and the $x$-axis. Further, we say that
  $(s,t)$ is an \emph{inner corner} of $\mathcal{C}(P)$, if the path touches
  the cell from two sides (the path in Figure~\ref{fig:path1} has the two
  inner corners $(2,1)$ and $(3,2)$) and it is a \emph{negative inner corner}
  if $s\leq 0$ or $t\leq 0$ (in Figure~\ref{fig:path2} there are the negative
  inner corners $(1,0)$ and $(2,-2)$ in the left path and $(-1,1)$ and
  $(-3,2)$ in the right path). For a cell $(s,t)$ we define $N(s,t)$ to be
  $1$ if it is a negative inner corner and $0$ otherwise.
Then, the definition of the weight of a path is equivalent to 
\begin{equation}\label{eqn:area}
  w(P)=\prod_{(i,j) \in \mathcal{C}(P)} (-1)^{N(i,j)} w(i,j)^{\mathrm{sgn}((i,j))}
\end{equation}
where $\mathrm{sgn}((i,j))\eqdef \mathrm{sgn}(i-1)\mathrm{sgn}(j-1)$.
Therefore, $w(\mathcal{P}_{\Z}((0,0)\rightarrow (k,n-k)))$ can be seen as
an area generating function. If a hybrid lattice path ends at $(k,n-k)$,
the exponents of the small weights in \eqref{eqn:area} are always $1$ if
$0\leq k \leq n$ and $-1$ otherwise. Moreover, the product over all
$(-1)^{N(i,j)}$ is $(-1)^k$ if $n<0\leq k$, it is $(-1)^{n-k}$ if
$k\leq n < 0$, and $1$ otherwise.
\end{remark}
\begin{remark}
  A recent preprint by O'Sullivan~\cite[Figure~1]{Sul} contains a figure
  similar to our Figure~\ref{fig:steps2}, showing the values of
  the \defn{harmonic multiset numbers}
  $||\begin{smallmatrix}n\\k\end{smallmatrix}||$
  (that extend the Stirling numbers) for integers $n$ and $k$.
\end{remark}  
\begin{example}
  By setting $w(s,t)=1$ for all $s,t \in \Z$, the generating function
  of hybrid lattice paths is given by the binomial coefficients with integer
  values defined by \eqref{def:bineq}. By setting $w(s,t)=q$ for all
  $s,t \in \Z$, we obtain a $q$-weighting of hybrid lattice paths with
  the generating function given by the $q$-binomial coefficients with
  integer values defined by \eqref{def:qbineq}.
\end{example}

\subsection{Reflection formulae}
As pointed out in \cite{Schl1}, the weight-dependent binomial coefficients
do not satisfy the symmetry $\wBinomText{n}{k}{w}=\wBinomText{n}{n-k}{w}$.
Nevertheless, we can prove the following reflection formula using the weight 
\begin{align}
\widehat{w}(s,t)\eqdef w_{y,x}(s,t)=w(t,s)^{-1} \label{def:hatweight}
\end{align}
from the involution \eqref{eqn:inv1} and 
\begin{align}
  \widehat{W}(s,t)\eqdef \prod_{j=1}^{t} \widehat{w}(s,j).
  \label{def:bighatweight}
\end{align}
\begin{theorem}\label{thm:refl1}
Let $n,k \in \Z$ and $\widehat{w}(s,t)=w(t,s)^{-1}$. Then,
$$\wBinom{n}{k}{w}=\wBinom{n}{n-k}{{\widehat{w}}} \prod_{j=1}^{k} W(j,n-k).$$
\end{theorem}
\begin{proof}
  We proceed by showing that both sides of the equation satisfy the same
  recursion and initial conditions.
	
  For $k=0$ and $n=k$ we obtain from \eqref{def:wbineq1} and the definition
  for products \eqref{prod} that
\begin{align*}
\wBinom{n}{n}{{\widehat{w}}} \prod_{j=1}^{0} W(j,n)&=1=\wBinom{n}{0}{w},\\
	\wBinom{n}{0}{{\widehat{w}}} \prod_{j=1}^{n} W(j,0)&=1=\wBinom{n}{n}{w}.
\end{align*}
For $(n+1,k) \neq (0,0)$, note that
$$\widehat{W}(n+1-k,k)=\prod_{j=1}^{k}\widehat{w}(n+1-k,j)
=\prod_{j=1}^{k}w(j,n+1-k)^{-1}$$ to deduce the recurrence relation
\begin{align*}
  &A_{n+1,k} \eqdef \wBinom{n+1}{n+1-k}{{\widehat{w}}}
    \prod_{j=1}^{k} W(j,n+1-k) \\
  &= \Bigg(\wBinom{n}{n+1-k}{{\widehat{w}}}
    + \wBinom{n}{n-k}{{\widehat{w}}} \widehat{W}(n+1-k,k)\Bigg)
    \prod_{j=1}^{k} W(j,n+1-k)\\
  &=\wBinom{n}{n-(k-1)}{{\widehat{w}}}
    \Bigg(\prod_{j=1}^{k-1} W(j,n-(k-1))\Bigg) W(k,n+1-k)
    +  \wBinom{n}{n-k}{{\widehat{w}}}\prod_{j=1}^{k} W(j,n-k)\\
		&=A_{n,k-1}W(k,n+1-k)+A_{n,k},
\end{align*}
which is equal to the recurrence relation \eqref{def:wbineq2}
for $\wBinom{n}{k}{w}$.
\end{proof}
Theorem~\ref{thm:refl1} also has a combinatorial interpretation in terms of hybrid lattice paths weighted by area \eqref{eqn:area}. The $w$-binomial coefficient on the right side of the equation corresponds to lattice paths from $(0,0)$ to $(n-k,k)$, where the weights are inverted and reflected by the line $y=x$, see the leftmost path in Figure~\ref{fig:reflection}~and~\ref{fig:reflection2}. If we reflect such a path with its weights by the line $y=x$, we obtain a lattice path to $(k,n-k)$, where the weight of the path corresponds to the unit squares between the path and the $y$-axis and all weights are still inverted compared to the usual weighting, see the middle path in Figure~\ref{fig:reflection}~and~\ref{fig:reflection2}. By multiplying with the weights $\prod_{j=1}^{k} W(j,n-k)$, which are the corresponding weights to the rectangle with vertices $(0,0)$, $(k,0)$, $(k,n-k)$ and $(0,n-k)$, the weights between the $y$-axis and the path cancel and we are left with the (non-inverted) weights between the path and the $x$-axis, see the rightmost path in Figure~\ref{fig:reflection}~and~\ref{fig:reflection2}. So we are left with a path from $(0,0)$ to $(k,n-k)$ which is weighted as in \eqref{eqn:area}, since the signs do not change. 
\begin{figure}
	\centering
	\scalebox{0.9}{
	\begin{tikzpicture}[scale=1.2]
		\path[fill=gray!20] (1,0) rectangle (2,1) node[pos=0.5,color=black] {\tiny$w(1,2)^{-1}$};
		\path[fill=gray!20] (2,0) rectangle (3,1) node[pos=0.5,color=black] {\tiny$w(1,3)^{-1}$};
		\path[fill=gray!20] (3,0) rectangle (4,1) node[pos=0.5,color=black] {\tiny$w(1,4)^{-1}$};
		\path[fill=gray!20] (2,1) rectangle (3,2) node[pos=0.5,color=black] {\tiny$w(2,3)^{-1}$};
		\path[fill=gray!20] (3,1) rectangle (4,2) node[pos=0.5,color=black] {\tiny$w(2,4)^{-1}$};
		
		\draw[->,>=latex] (0,-0.3) -- (0,2.3);
		\draw[->,>=latex] (-0.3,0) -- (4.3,0);
		\draw[color=myred] (-0.3,-0.3) -- (2.3,2.3);
		
		\foreach \x in {0,...,4}
			\draw[dotted] (\x,0)--(\x,2);
		\foreach \y in {0,...,2}
			\draw[dotted] (0,\y)--(4,\y);
			
		\draw[line width=2pt, myblue] (0,0) -- (1,0) -- (1,1)--(2,1) -- (2,2)--(3,2)--(4,2);	
		
		\foreach \x in {0,...,4}
		\foreach \y in {0,...,2}
			\draw[fill=black](\x,\y)circle(2pt);
		
		\draw[thick,->,>=latex] (4.5,2) -- node [midway,above] {\tiny Reflect} (5.5,2);
		
		\path[fill=gray!20] (6,1) rectangle (7,2) node[pos=0.5,color=black] {\tiny$w(1,2)^{-1}$};
		\path[fill=gray!20] (6,2) rectangle (7,3) node[pos=0.5,color=black] {\tiny$w(1,3)^{-1}$};
		\path[fill=gray!20] (6,3) rectangle (7,4) node[pos=0.5,color=black] {\tiny$w(1,4)^{-1}$};
		\path[fill=gray!20] (7,2) rectangle (8,3) node[pos=0.5,color=black] {\tiny$w(2,3)^{-1}$};
		\path[fill=gray!20] (7,3) rectangle (8,4) node[pos=0.5,color=black] {\tiny$w(2,4)^{-1}$};
		
		\draw[->,>=latex] (6,-0.3) -- (6,4.3);
		\draw[->,>=latex] (5.7,0) -- (8.3,0);
		\draw[color=myred] (5.7,-0.3) -- (8.1,2.1);
		
		\foreach \x in {6,...,8}
			\draw[dotted] (\x,0)--(\x,4);
		\foreach \y in {0,...,4}
			\draw[dotted] (6,\y)--(8,\y);
		
		\draw[line width=2pt, myblue] (6,0) -- (6,1) -- (7,1)--(7,2) -- (8,2)--(8,3)--(8,4);	
		
		\foreach \x in {6,...,8}
		\foreach \y in {0,...,4}
			\draw[fill=black](\x,\y)circle(2pt);
		
		\draw[thick,->,>=latex] (8.5,2) -- node [midway,above] {\tiny $\displaystyle \prod_{j=1}^{k} W(j,n-k)$ } (9.5,2);
		
		\path[fill=gray!20] (10,0) rectangle (11,1) node[pos=0.5,color=black] {\tiny$w(1,1)$};
		\path[fill=gray!20] (11,0) rectangle (12,1) node[pos=0.5,color=black] {\tiny$w(2,1)$};
		\path[fill=gray!20] (11,1) rectangle (12,2) node[pos=0.5,color=black] {\tiny$w(2,2)$};
		
		\draw[->,>=latex] (10,-0.3) -- (10,4.3);
		\draw[->,>=latex] (9.7,0) -- (12.3,0);
		
		\foreach \x in {10,...,12}
			\draw[dotted] (\x,0)--(\x,4);
		\foreach \y in {0,...,4}
			\draw[dotted] (10,\y)--(12,\y);
		
		\draw[line width=2pt, myblue] (10,0) -- (10,1) -- (11,1)--(11,2) -- (12,2)--(12,3)--(12,4);	
		
		\foreach \x in {10,...,12}
		\foreach \y in {0,...,4}
			\draw[fill=black](\x,\y)circle(2pt);
	\end{tikzpicture}
}
	\caption{An illustration of the combinatorial interpretation of Theorem~\ref{thm:refl1} for $0 \leq k \leq n$.}
	\label{fig:reflection}
\end{figure}
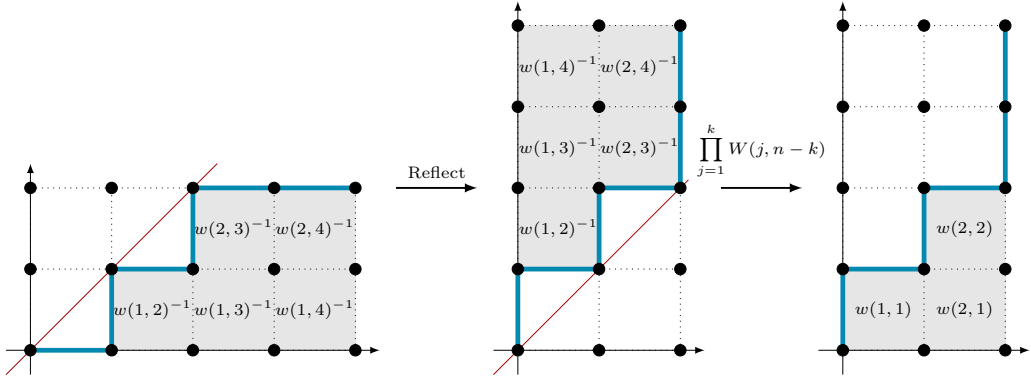

\begin{figure}
	\centering
	\scalebox{0.9}{
	\begin{tikzpicture}[scale=1.2]
		\path[fill=gray!20] (0,-1) rectangle +(1,1) node[pos=0.5,color=black] {\tiny$w(0,1)^{\scalebox{0.8}[1]{-}1}$};
		\path[fill=gray!20] (1,-1) rectangle +(1,1) node[pos=0.5,color=black] {\tiny$w(0,2)^{\scalebox{0.8}[1]{-}1}$};
		\path[fill=gray!20] (0,-2) rectangle +(1,1) node[pos=0.5,color=black] {\tiny$w(\scalebox{0.8}[1]{-}1,1)^{\scalebox{0.8}[1]{-}1}$};
		\path[fill=gray!20] (1,-2) rectangle +(1,1) node[pos=0.5,color=black] {\tiny$w(\scalebox{0.8}[1]{-}1,2)^{\scalebox{0.8}[1]{-}1}$};
		\path[fill=gray!20] (1,-3) rectangle +(1,1) node[pos=0.5,color=black] {\tiny$w(\scalebox{0.6}{-}2,2)^{\scalebox{0.6}{-}1}$};
		
		\draw[->,>=latex] (0,-4.3) -- (0,0.3);
		\draw[->,>=latex] (-0.3,0) -- (2.3,0);
		\draw[color=myred] (-0.5,-0.5) -- (0.5,0.5);
	
		\foreach \x in {0,...,2}
			\draw[dotted] (\x,0)--(\x,-4);
		\foreach \y in {-4,...,0}
			\draw[dotted] (0,\y)--(2,\y);
		
		\draw[line width=2pt, myblue] (0,0)--(0,-2)--(1,-2)--(1,-3)--(2,-3)--(2,-4);	
		
		\foreach \x in {0,...,2}
		\foreach \y in {-4,...,0}
			\draw[fill=black](\x,\y)circle(2pt);
		
		\draw[thick,->,>=latex] (2.3,-2) -- (2.7,-2);
		
		\path[fill=gray!20] (6,-3) rectangle +(1,1) node[pos=0.5,color=black] {\tiny$w(0,1)$};
		\path[fill=gray!20] (6,-2) rectangle +(1,1) node[pos=0.5,color=black] {\tiny$w(0,2)$};
		\path[fill=gray!20] (5,-3) rectangle +(1,1) node[pos=0.5,color=black] {\tiny$w(\scalebox{0.8}[1]{-}1,1)$};
		\path[fill=gray!20] (5,-2) rectangle +(1,1) node[pos=0.5,color=black] {\tiny$w(\scalebox{0.8}[1]{-}1,2)$};
		\path[fill=gray!20] (4,-2) rectangle +(1,1) node[pos=0.5,color=black] {\tiny$w(\scalebox{0.8}[1]{-}2,2)$};
		
		\draw[->,>=latex] (7,-3.3) -- (7,-0.7);
		\draw[->,>=latex] (2.7,-3) -- (7.3,-3);
		\draw[color=myred] (6.5,-3.5) -- (7.5,-2.5);
		
		\foreach \x in {3,...,7}
			\draw[dotted] (\x,-3)--(\x,-1);
		\foreach \y in {-3,...,-1}
			\draw[dotted] (3,\y)--(7,\y);
		
		\draw[line width=2pt, myblue] (7,-3) -- (5,-3) -- (5,-2) -- (4,-2) -- (4,-1) -- (3,-1);	
		
		\foreach \x in {3,...,7}
		\foreach \y in {-3,...,-1}
			\draw[fill=black](\x,\y)circle(2pt);
		
		\draw[thick,->,>=latex] (7.3,-2) -- (7.7,-2);
		
		\path[fill=gray!20] (9,-3) rectangle +(1,1) node[pos=0.5,color=black] {\tiny$w(\scalebox{0.8}[1]{-}2,1)^{\scalebox{0.8}[1]{-}1}$};
		\path[fill=gray!20] (8,-2) rectangle +(1,1) node[pos=0.5,color=black] {\tiny$w(\scalebox{0.8}[1]{-}3,2)^{\scalebox{0.8}[1]{-}1}$};
		\path[fill=gray!20] (8,-3) rectangle +(1,1) node[pos=0.5,color=black] {\tiny$w(\scalebox{0.8}[1]{-}3,1)^{\scalebox{0.8}[1]{-}1}$};
		
		\draw[->,>=latex] (12,-3.3) -- (12,-0.7);
		\draw[->,>=latex] (7.7,-3) -- (12.3,-3);
		
		\foreach \x in {8,...,12}
			\draw[dotted] (\x,-3)--(\x,-1);
		\foreach \y in {-3,...,-1}
			\draw[dotted] (8,\y)--(12,\y);
		
		\draw[line width=2pt, myblue] (12,-3) -- (10,-3) -- (10,-2) -- (9,-2) -- (9,-1) -- (8,-1);	
		
		\foreach \x in {8,...,12}
		\foreach \y in {-3,...,-1}
			\draw[fill=black](\x,\y)circle(2pt);
	\end{tikzpicture}
}
	\caption{An illustration of the combinatorial interpretation of Theorem~\ref{thm:refl1} for $k \leq n <0$.}
	\label{fig:reflection2}
\end{figure}
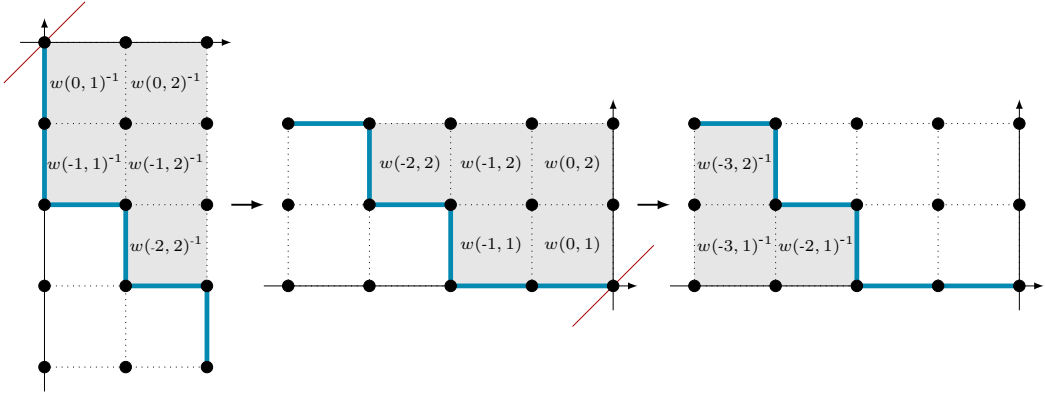

In the case of $w(s,t)=q$, Theorem~\ref{thm:refl1} reduces to the
well-known formula
$$\eBinom{n}{k}{q}=\eBinom{n}{n-k}{q^{-1}} q^{k(n-k)}.$$
In the following theorem we give two reflection formulae from which
we can deduce the behaviour of $w$-binomial coefficients with negative
values (see Proposition~\ref{prop:six}). Recall the involutions
\eqref{eqn:inv3} and \eqref{eqn:inv6} to define the dual weight functions
\begin{align}
  \widetilde{w}(s,t)&\eqdef w_{x^{-1},x^{-1}y}(s,t)=w(1-s-t,t)^{-1},
                      \label{def:tildeweight}\\
  \breve{w}(s,t)&\eqdef w_{y^{-1}x,y^{-1}}(s,t)
                  =w(s,1-s-t)^{-1}. \label{def:breveweight}
\end{align}
The proof of both equations can be done analogously to the proof of
Theorem~\ref{thm:refl1}; we therefore omit the proof.

\begin{theorem}\label{thm:refl2}
Let $n,k \in \Z$. Then,
\begin{align}
  \wBinom{n}{k}{w}
  &= (-1)^{n-k} \mathrm{sgn} (n-k)
   \wBinom{-k-1}{-n-1}{{\widetilde{w}}}
    \prod_{j=1}^{n-k} W(n+1-j,j)^{-1} \label{eqn:refl2a} \\
  &= (-1)^{k} \mathrm{sgn} (k) \wBinom{k-n-1}{k}{{\breve{w}}}
    \prod_{j=1}^{k} W(j,-j) \label{eqn:refl2b}
\end{align}
where $\widetilde{w}(s,t)=w(1-s-t,t)^{-1}$, $\breve{w}(s,t)=w(s,1-s-t)^{-1}$
and $\mathrm{sgn}(k)$ is defined by \eqref{def:sgn}.
\end{theorem}
By a careful case-by-case analysis one can also find combinatorial interpretations of the formulas in Theorem~\ref{thm:refl2} similar to the interpretation of Theorem~\ref{thm:refl1} above. We leave the details to the reader.

Given these reflection formulae, we can give a weight-dependent
generalization of a Proposition by Loeb~\cite[Proposition 4.1]{L}.
\begin{proposition}[The six regions]\label{prop:six}
  Let $n,k \in \Z$, $\widehat{w}(s,t)=w(t,s)^{-1}$,
  $\widetilde{w}(s,t)=w(1-s-t,t)^{-1}$ and $\breve{w}(s,t)=w(s,1-s-t)^{-1}$.
  Depending on the signs of $n$ and $k$, the following formulae apply:
\begin{enumerate}
\item If $0 \leq k \leq n$, then $\wBinomText{n}{k}{w}=\wBinomText{n}{k}{w}$.
\item If $n < 0 \leq k$, then
  $\wBinomText{n}{k}{w}=(-1)^{k} \wBinomText{k-n-1}{k}{{\breve{w}}}
  \prod_{j=1}^{k} W(j,-j)$.
\item If $k \leq n < 0$, then
  $\wBinomText{n}{k}{w}=(-1)^{n-k} \wBinomText{-k-1}{-n-1}{{\widetilde{w}}}
  \prod_{j=1}^{n-k} W(n+1-j,j)^{-1}$.
\item If $0 \leq n < k$, then $\wBinomText{n}{k}{w}=0$.
\item If $n<k<0$, then $\wBinomText{n}{k}{w}=0$.
\item If $k<0\leq n$, then $\wBinomText{n}{k}{w}=0$.
\end{enumerate}
\end{proposition} 
Case $(1)$ corresponds to ordinary weighted lattice paths. The cases
$(2)$ and $(3)$ show that weighted hybrid lattice paths from $(0,0)$
to $(n,m)$, where $n$ or $m$ is negative, can be realized as ordinary
weighted lattice paths where the weights $w(s,t)$ are replaced by
$w(s,1-s-t)^{-1}$ or $w(1-s-t,t)^{-1}$ and additionally, the weight of
the path is multiplied by some weight given in the above equations.
Cases $(4)-(6)$ are already discussed in the proof of Theorem~\ref{thm:hlp}.

\subsection{An extension of the noncommutative binomial theorem}
For the purpose of the following theorem, we consider the algebras of
formal power series $\C_w[[x,x^{-1},y]]$ and $\C_w[[x,y,y^{-1}]]$ of formal
power series. Suppose $f_n(x,y)$ is a function with power series expansions
\begin{subequations} \label{eqn:power}
\begin{align}
  f_n(x,y)=\sum_{k\geq 0} a_k x^k y^{n-k}
  & \quad \text{in $\C_w[[x,y,y^{-1}]]$} \\
\intertext{or}
f_n(x,y)=\sum_{k\leq n} b_{k} x^{k} y^{n-k} &\quad \text{in $\C_w[[x,x^{-1},y]]$}.
\end{align}
\end{subequations}
Following \cite{FS}, we extract coefficients of the expansions by writing
$[x^ky^{n-k}]f_n(x,y) = a_k$ for $k\geq0$ and $[x^ky^{n-k}]f_n(x,y) = b_k$
for $k<0$. Now we can state the weight-dependent noncommutative binomial
theorem for integer values which generalizes the results of \cite{FS},
\cite{L} and \cite{Schl1}.

\begin{theorem}\label{thm:extbinomial}
  Let $n,k\in \Z$ and $x,y$ be invertible variables satisfying the
  commutation relations \eqref{def:Cwxy3}, then we have
\begin{equation}\label{eqn:extbinomial}
[x^k y^{n-k}](x+y)^n = \wBinom{n}{k}{w}.
\end{equation}
\end{theorem}
\begin{proof}
We will prove the theorem by showing that for $n \in \Z$,
\begin{align}
  (x+y)^n &= \sum_{k\geq0} \wBinom{n}{k}{w} x^k y^{n-k} \quad
            \text{in $\C_w[[x,y,y^{-1}]]$}  \label{eqn:extbin1} \\
  \intertext{or}
  (x+y)^n &= \sum_{k\leq n} \wBinom{n}{k}{w} x^{k} y^{n-k} \quad
            \text{in $\C_w[[x,x^{-1},y]]$}. \label{eqn:extbin2}
\end{align}
We begin with \eqref{eqn:extbin1}. The case $n\geq0$ is just
Theorem \ref{thm:binomial} combined with the fact that
$\wBinomText{n}{k}{w}=0$ if $0\leq n < k$. We prove the case $n<0$
by induction. 
	
For $n=-1$, we use the geometric series and apply the
relations \eqref{def:Cwxy3} to obtain
\begin{align*}
  (x+y)^{-1}&=y^{-1}(xy^{-1}+1)^{-1}=y^{-1} \sum_{k \geq 0} (-1)^k (xy^{-1})^k\\
 & = \sum_{k\geq0} (-1)^k \Big( \prod_{j=1}^{k} W(j,-j) \Big) x^k y^{-k-1}.
\end{align*}
From Example~\ref{ex:minus1} we know that for $k\geq0$
$$(-1)^k \prod_{j=1}^{k} W(j,-j)= \wBinom{-1}{k}{w}.$$
Suppose \eqref{eqn:extbin1} holds for some $n+1 \leq -1$. Then
\begin{align*}
(x+y)^{n+1}&=\sum_{k\geq0} \wBinom{n+1}{k}{w} x^k y^{n+1-k} \\
           &=\sum_{k\geq0} \wBinom{n}{k}{w} x^k y^{n+1-k}
             + \sum_{k\geq0} \wBinom{n}{k-1}{w} W(k,n+1-k) x^k y^{n+1-k}\\
           &=\sum_{k\geq0} \wBinom{n}{k}{w} x^k y^{n+1-k}
             + \sum_{k\geq0} \wBinom{n}{k}{w} W(k+1,n-k) x^{k+1} y^{n-k}\\
	&=\Big( \sum_{k\geq 0} \wBinom{n}{k}{w} x^k y^{n-k} \Big) (x+y)
\end{align*}
Hence,  \eqref{eqn:extbin1} is true for $n$ as well and therefore, it is
true for all $n\in \Z$.  Since we pulled out $y^{-1}$ in the initial step,
the expansion is true in $\C_w[[x,y,y^{-1}]]$.
	
In order to prove the second expansion \eqref{eqn:extbin2} we apply
the involution~\eqref{eqn:inv1} to \eqref{eqn:extbin1} to obtain
\begin{equation*}
(y+x)^n = \sum_{k \geq 0} \wBinom{n}{k}{{\widehat{w}}} y^k x^{n-k},
\end{equation*}
where $\widehat{w}(s,t)=w(t,s)^{-1}$. We use the reflection formula of
Theorem~\ref{thm:refl1} to obtain
\begin{align*}
  (x+y)^n &= (y+x)^n = \sum_{k \geq 0} \wBinom{n}{k}{{\widehat{w}}} y^k x^{n-k}
            =\sum_{k \geq 0} \wBinom{n}{k}{{\widehat{w}}}
            \Big( \prod_{i=1}^{n-k} W(i,k) \Big) x^{n-k} y^{k}\\
          &=\sum_{k \geq 0} \wBinom{n}{n-k}{w} x^{n-k} y^{k}
            =\sum_{k\leq n} \wBinom{n}{k}{w} x^{k} y^{n-k}
\end{align*}
in $\C_{w}[[x,x^{-1},y]]$, as claimed.
\end{proof}

Recall the interpretation of expansion \eqref{eqn:extbin1} in the case
$n\geq 0$ in terms of weighted lattice paths in the paragraph before
Theorem~\ref{thm:binomial}. It turns out that this interpretation is also
valid for the case $n<0$ and for expansion \eqref{eqn:extbin2}. In the
following we will extend the details of this interpretation.

In Theorem~\ref{thm:hlp} we interpreted the $w$-binomial coefficients as
the weighted sum over hybrid lattice paths. Identify expressions (or words)
in $\C_w[x,x^{-1},y,y^{-1}]$ with hybrid lattice paths locally (variable by
variable, or step by step) as follows:
\begin{align*}
x & \leftrightarrow \text{east step} \\
x^{-1} & \leftrightarrow \text{west step} \\
y & \leftrightarrow \text{north step} \\
y^{-1} & \leftrightarrow \text{south step} 
\end{align*}
This means, reading from left to right, that $xy^{-1}$ corresponds to a
path where an east step is followed by a south step and $y^{-1}x$ corresponds
to a path where the south step is followed by an east step. The relations
of the algebra $\C_w[x,x^{-1},y,y^{-1}]$ take into account the changes of
the respective weights when two consecutive steps are being interchanged.
For instance, the lattice path $P_0$ from the left side of
Figure~\ref{fig:path2} corresponds to the following expression:
\begin{align*}
&y^{-1}xy^{-1}y^{-1}xy^{-1} = w(1,0)^{-1} x y^{-2} w(1,0)^{-1} x y^{-2} \\
  &=w(1,0)^{-1} w(2,-2)^{-1} w(2,-1)^{-1} w(2,0)^{-1} x^2 y^{-4}
    = (-1)^2 w(P_0) x^2 y^{-4}
\end{align*}
where $w(P_0)$ is the weight of the path $P_0$ assigned to hybrid lattice
paths in Theorem~\ref{thm:hlp}.

In summary, the expression $(x+y)^n$ translates into the sum over all
expressions in $\C_w[x,x^{-1},y,y^{-1}]$ corresponding to a hybrid lattice
path from $(0,0)$ to $(k,n-k)$, where $k$ can be any nonnegative integer
(in \eqref{eqn:extbin1}) or any integer $\leq n$ (in \eqref{eqn:extbin2}),
and the expression is multiplied by
$\epsilon=\mathrm{sgn}(n)^k \mathrm{sgn}(k)^n$.

\subsection{Convolution formulae}\label{sec:conv}
In~\cite{Schl1}, the second author derived a weight-dependent
generalization of the Chu--Vandermonde convolution formula 
\[
	\binom{n+m}{k} = \sum_{j=0}^{k} \binom{n}{j} \binom{m}{k-j}
\]
by expanding $(x+y)^{n+m}$ on the one hand and $(x+y)^n(x+y)^m$ on the
other hand using Theorem~\ref{thm:binomial} and comparing the coefficients
of $x^ky^{n+m-k}$ for $n,m,k\geq0$. Given Theorem~\ref{thm:extbinomial},
we can expand $(x+y)^{n+m}$ and $(x+y)^n(x+y)^m$ using \eqref{eqn:extbin1}
for all $n,m \in \Z$ to obtain the following weight-dependent convolution
formula by again comparing the coefficients of $x^ky^{n+m-k}$ for $k\geq 0$.
The proof is similar to the proof of \cite[Corollary~1]{Schl1} and
therefore we omit the details.
\begin{corollary}\label{cor:wdconv1}
 Let $n,m \in \Z$ and $k\geq 0$.
 For the $w$-binomial coefficients in \eqref{def:wbineq}, defined by the
 doubly-indexed sequence of invertible variables $(w(s,t))_{s,t\in\Z}$,
 we have the following formal identity in $\C[(w(s,t))_{s,t\in\Z}]$:
 \begin{equation}\label{eqn:wdconv1id}
 	\wBinom{n+m}{k}{w}=\sum_{j=0}^{k}
 	\wBinom{n}{j}{w}
 	\left(x^jy^{n-j}\,
 	\wBinom{m}{k-j}{w}
 	y^{j-n}x^{-j}\right)
 	\prod_{i=1}^{k-j}W(i+j,n-j).
 \end{equation}
\end{corollary}

Compared to the corresponding identity in~\cite[Corollary 1]{Schl1}, the sum
in the above identity is bounded by $k$ instead of $min(k,n)$. But since
$\wBinomText{n}{j}{w}$ is zero if $j>n\geq 0$, this makes no difference in
the case $n \geq 0$. The above identity is formal because it contains
noncommuting variables $x$ and $y$ defined by \eqref{eqn:noncommrel3}, which
can be understood to be shift operators shifting the weight-dependent binomial
coefficient in the expression
$$x^jy^{n-j}\,\wBinom{m}{k-j}{w}y^{j-n}x^{-j}.$$
Evaluating this expression will shift all weights in $\wBinomText{m}{k-j}{w}$
and afterwards $x^jy^{n-j}$ will cancel with $y^{j-n}x^{-j}$ so we obtain an
identity in $\C[(w(s,t))_{s,t\in\Z}]$ since all $x$ and $y$ vanish. 

By expanding $(x+y)^{n+m}$ and $(x+y)^n(x+y)^m$ using \eqref{eqn:extbin2}
and comparing coefficients as before, we obtain a second $w$-Chu--Vandermonde
convolution formula.
\begin{corollary}\label{cor:wdconv2}
 Let $n,m \in \Z$ and $k \leq n+m$.
 For the $w$-binomial coefficients in \eqref{def:wbineq}, defined by the
 doubly-indexed sequence of invertible variables $(w(s,t))_{s,t\in\Z}$,
 we have the following formal identity in $\C[(w(s,t))_{s,t\in\Z}]$:
 \begin{equation}\label{eqn:wdconv2id}
 	\wBinom{n+m}{k}{w}=\sum_{j=k-m}^{n}
 	\wBinom{n}{j}{w}
 	\left(x^jy^{n-j}\,
 	\wBinom{m}{k-j}{w}
 	y^{j-n}x^{-j}\right)
 	\prod_{i=1}^{k-j}W(i+j,n-j).
 \end{equation}
\end{corollary}
If $n+m<k<0$, both sides of the equation vanish, therefore
\eqref{eqn:wdconv2id} is true for all $k<0$. This equation generalizes an
identity in \cite[Lemma 4.9]{FS}. There the corresponding identity is
limited to the case $n,m,k<0$ whereas Corollary~\ref{cor:wdconv2} is even
true if $n,m$ are positive or have mixed signs.

In fact, Corollary~\ref{cor:wdconv1} and \ref{cor:wdconv2} are equivalent. To
see the correspondence, let $k\leq n+m$ and apply Corollary~\ref{cor:wdconv1}
to $\wBinomText{n+m}{n+m-k}{w}$. Additionally, we apply the involution from
\eqref{eqn:inv1} with $\widehat{w}(s,t)=w(t,s)^{-1}$ to obtain
\begin{align*}
	&\wBinom{n+m}{n+m-k}{{\widehat{w}}}\\
	&=\sum_{j=0}^{n+m-k} \wBinom{n}{j}{{\widehat{w}}}
 	\left(y^j x^{n-j}\,
 	\wBinom{m}{n+m-k-j}{{\widehat{w}}}
 	x^{j-n}y^{-j}\right)
 	\prod_{i=1}^{n+m-k-j}\widehat{W}(i+j,n-j)\\
 	&=\sum_{j=0}^{n+m-k} \wBinom{n}{n+m-k-j}{{\widehat{w}}}
 	\left(x^{-m+k+j} y^{n+m-k-j}\,
 	\wBinom{m}{j}{{\widehat{w}}}
 	y^{-n-m+k+j}x^{m-k-j}\right)
 	\\
 	&\quad \times \prod_{i=1}^{j}\widehat{W}(i+n+m-k-j,-m+k+j)\\
 	&=\sum_{j=k-m}^{n} \wBinom{n}{n-j}{{\widehat{w}}}
 	\left(x^{j} y^{n-j}\,
 	\wBinom{m}{m-k+j}{{\widehat{w}}}
 	y^{j-n}x^{-j}\right)
 	\prod_{i=1}^{m-k+j}\widehat{W}(i+n-j,j).
 \end{align*}
 Then we apply the reflection formula from Theorem~\ref{thm:refl1} on both
 sides to obtain
\begin{align*}
&\wBinom{n+m}{k}{w} \left( \prod_{i=1}^{k}W(i,n+m-k)^{-1} \right)
   =\sum_{j=k-m}^{n} \wBinom{n}{j}{w}
   \left( \prod_{i=1}^{j}W(i,n-j)^{-1} \right) \\
&\times\left(x^j y^{n-j}\,
 	\wBinom{m}{k-j}{w} \left( \prod_{i=1}^{m-k+j}W(i,m-k+j)^{-1} \right)
 	y^{j-n}x^{-j}\right)
 	\prod_{i=1}^{m-k+j}\widehat{W}(i+n-j,j)
\end{align*}
It is a straightforward calculation to see that this expression is
equivalent to Corollary~\ref{cor:wdconv2}.

We can apply the same method to the second and third weight-dependent
binomial convolution formula in \cite{Schl1} to see that these two identities
(extended to integers) are also equivalent to Corollary~\ref{cor:wdconv1}.
There we can use the reflection formulae from Theorem~\ref{thm:refl2} and
the involutions from \eqref{eqn:inv3} and \eqref{eqn:inv6}. The equivalence of
the three convolution formulae in \cite{Schl1} also explains why the elliptic
case of all three formulae are variants of the same elliptic hypergeometric
summation formula, namely Frankel and Turaev's ${}_{10}V_9$ summation. 

\subsection{Combinatorial interpretation}\label{sec:ref_comb}
In Theorem~\ref{thm:hlp} we interpret $w$-binomial coefficients in terms of
lattice paths. If the signs of $n$ and $m$ are identical,
Corollaries~\ref{cor:wdconv1}~and~\ref{cor:wdconv2} translate into
convolutions of paths. If $n,m \geq 0$, both identities translate to
convolutions corresponding to a diagonal in the first quadrant (see
Figure~\ref{fig:convolution1}). A weighted path counted by
$\wBinomText{n+m}{k}{w}$, which crosses the point $(j,n-j)$ for some
$0 \leq j \leq k$, can be decomposed into three parts: A weighted path from
$(0,0)$ to $(j,n-j)$ (counted by $\wBinomText{n}{j}{w}$), a weighted path
from $(j,n-j)$ to $(k,n+m-k)$ (counted by
$x^j y^{n-j} \wBinomText{m}{k-j}{w} y^{j-n} x^{-j}$) and the weights
corresponding to the rectangle between the $x$-axis and the second path
(which contributes $\prod_{i=1}^{k-j}W(i+j,n-j)$). By summing over all
$0 \leq j \leq k$ we obtain the convolution.
\begin{center}
\begin{figure}
\begin{minipage}{0.32\textwidth}
	\centering
	
	\begin{tikzpicture}[scale=0.6]
		\draw[->,>=latex] (0,-0.5) -- (0,7.5);
		\draw[->,>=latex] (-0.5,0) -- (6.5,0);
		\draw[color=myblue, line width=2pt] (0,0) -- (1,0) -- (1,1) -- (1,2) -- (2,2) -- (2,3) -- (2,4) -- (3,4) -- (4,4) -- (4,6);
		\foreach \x in {0,...,6}
		\foreach \y in {0,...,7}
		{\draw[fill=black](\x,\y)circle(2pt);}
		\node[above left=-0.1, color=black] at (0,0) {\footnotesize $(0,0)$};
		\node[above right=-0.1, color=black] at (4,6) {\footnotesize $(k,n+m-k)$};
		\node[above left=-0.1, color=black] at (0,5) {\footnotesize $(0,n)$};
		\node[below right=-0.1, color=black] at (4,1) {\footnotesize $(k,n-k)$};
		\node[above right=-0.1, color=black] at (2,3) {\footnotesize \color{myblue} $(j,n-j)$};
		\draw[fill=black,color=black](0,0)circle(4pt);
		\draw[fill=black,color=black](4,6)circle(4pt);
		\draw[fill=black,color=black](0,5)circle(4pt);
		\draw[fill=black,color=black](4,1)circle(4pt);
		\draw (0,5) -- (4,1);
		\draw[fill=myblue,color=myblue](2,3)circle(4pt);
	\end{tikzpicture}

\end{minipage}
\hfill
\begin{minipage}{0.1\textwidth} \centering
	$\leftrightarrow$
\end{minipage}
\hfill
\begin{minipage}{0.55\textwidth}\centering
	\begin{tikzpicture}[scale=0.6]
	\path[fill=green!20] (2,0) rectangle (4,3) node[pos=.5, color=black] {};
	\path[pattern=horizontal lines, pattern color=green!30] (4,0) rectangle (4,3) node[pos=.5, color=black] {};
	\path[fill=myblue!10] (0,0) rectangle (2,3) node[pos=.5, color=black] {};
	\path[pattern=north west lines, pattern color=myblue!30] (0,0) rectangle (2,3) node[pos=.5, color=black] {};
	\path[fill=myred!10] (2,3) rectangle (4,6) node[pos=.5, color=black] {};
	\path[pattern=north east lines, pattern color=myred!30] (2,3) rectangle (4,6) node[pos=.5, color=black] {};
	\draw[color=myblue, line width=2pt] (0,0) -- (1,0) -- (1,1) -- (1,2) -- (2,2) -- (2,3);
	\draw[color=myred, line width=2pt] (2,3) -- (2,4) -- (3,4) -- (4,4) -- (4,6);
	\draw[fill=black,color=black](0,0)circle(3pt);
	\draw[fill=black,color=black](4,6)circle(3pt);
	\draw[fill=myblue,color=black](2,3)circle(3pt);
	\node[right, color=black] at (4,1.5) {\footnotesize ${\displaystyle \prod_{i=1}^{k-j}W(i+j,n-j)}$};
	\node[right, color=black] at (4,4.5) {\footnotesize $x^j y^{n-j} \wBinom{m}{k-j}{w} y^{j-n} x^{-j}$};
	\node[left, color=black] at (0,1.5) {\footnotesize $\wBinom{n}{j}{w}$};
	\end{tikzpicture}
\end{minipage}
\caption{For $n,m \geq 0$, Corollary~\ref{cor:wdconv1} translates into
  a convolution of paths.}
	\label{fig:convolution1}
\end{figure}
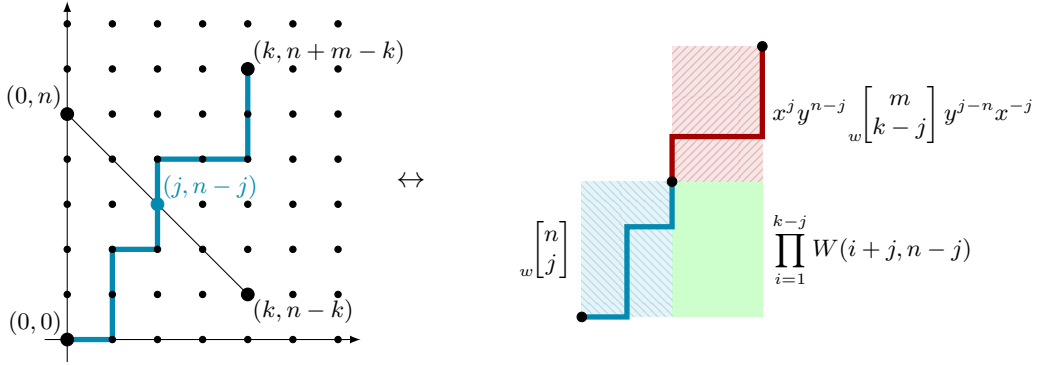
\end{center}
If $n,m < 0$, Corollary~\ref{cor:wdconv1} translates to a convolution
corresponding to a diagonal in the fourth quadrant and
Corollary~\ref{cor:wdconv2} to a diagonal in the second quadrant
(see Figure~\ref{fig:convolution2}). The decomposition is similar
to the $n,m\geq 0$ case.
\begin{center}
\begin{figure}
	\begin{minipage}{0.3\textwidth}
		\centering
	\begin{tikzpicture}[scale=0.6]
			\path[fill=green!20] (2,0) rectangle (4,-3) node[pos=.5, color=black] {};
			\path[fill=myblue!10] (0,0) rectangle (2,-3) node[pos=.5, color=black] {};
			\path[pattern=north west lines, pattern color=myblue!30] (0,0) rectangle (2,-3) node[pos=.5, color=black] {};
			\path[fill=myred!10] (2,-3) rectangle (4,-6) node[pos=.5, color=black] {};
			\path[pattern=north east lines, pattern color=myred!30] (2,-3) rectangle (4,-6) node[pos=.5, color=black] {};
			\draw[->,>=latex] (0,0.5) -- (0,-7.5);
			\draw[->,>=latex] (-0.5,0) -- (6.5,0);
			\draw (0,-5) -- (4,-1);
			\draw[color=myblue,line width=2pt] (0,0) -- (0,-1) -- (1,-1) -- (1,-2) -- (2,-2) -- (2,-3); 
			\draw[color=myred, line width=2pt] (2,-3)-- (2,-4) -- (3,-4) -- (3,-5) -- (4,-5) -- (4,-6);
			\foreach \x in {0,...,6}
			\foreach \y in {-7,...,0}
			{\draw[fill=black](\x,\y)circle(2pt);}
			\draw[fill=myblue,color=black](2,-3)circle(4pt);
			\node[below left=-0.1, color=black] at (0,0) {\footnotesize $(0,0)$};
			\node[below right=-0.1, color=black] at (4,-6) {\footnotesize $(k,n+m-k)$};
			\node[below left=-0.1, color=black] at (0,-5) {\footnotesize $(0,n)$};
			\node[above right=-0.1, color=black] at (4,-1) {\footnotesize $(k,n-k)$};
			\draw[fill=black,color=black](0,0)circle(4pt);
			\draw[fill=black,color=black](4,-6)circle(4pt);
			\draw[fill=black,color=black](0,-5)circle(4pt);
			\draw[fill=black,color=black](4,-1)circle(4pt);
			\node[below right=-0.1, color=black] at (2,-3) {\footnotesize \color{black} $(j,n-j)$};
		\end{tikzpicture}
	\end{minipage}
	\hfill
	\begin{minipage}{0.65\textwidth}\centering
		\begin{tikzpicture}[scale=0.6]
			\path[fill=green!20] (-5,0) rectangle (-8,3) node[pos=.5, color=black] {};
			\path[fill=myblue!10] (0,0) rectangle (-5,3) node[pos=.5, color=black] {};
			\path[pattern=north west lines, pattern color=myblue!30] (0,0) rectangle (-5,3) node[pos=.5, color=black] {};
			\path[fill=myred!10] (-5,3) rectangle (-8,4) node[pos=.5, color=black] {};
			\path[pattern=north east lines, pattern color=myred!30] (-5,3) rectangle (-8,4) node[pos=.5, color=black] {};
			\draw[->,>=latex] (0,-0.5) -- (0,6.5);
			\draw[->,>=latex] (0.5,0) -- (-9.5,0);
			\draw[color=myblue,line width=2pt] (0,0) -- (-1,0) -- (-1,1) -- (-2,1) -- (-2,2) -- (-3,2) -- (-4,2) -- (-4,3) -- (-5,3);
			\draw[color=myred,line width=2pt] (-5,3) -- (-6,3) -- (-7,3) -- (-7,4) -- (-8,4) ;
			\foreach \x in {-9,...,0}
			\foreach \y in {0,...,6}
			{\draw[fill=black](\x,\y)circle(2pt);}
			\node[below right=-0.1, color=black] at (0,0) {\footnotesize $(0,0)$};
			\node[above right=-0.1, color=black] at (-6,4) {\footnotesize $(k-m,n+m-k)$};
			\node[below right=-0.1, color=black] at (-2,0) {\footnotesize $(n,0)$};
			\node[above=-0.08, color=black] at (-8,4) {\footnotesize $(k,n+m-k)$};
			\node[above right=-0.1, color=black] at (-5,3) {\footnotesize \color{black}$(j,n-j)$};
			\draw[fill=black,color=black](0,0)circle(4pt);
			\draw[fill=black,color=black](-8,4)circle(4pt);
			\draw (-2,0) -- (-6,4);
			\draw[fill=black,color=black](-2,0)circle(4pt);
			\draw[fill=black,color=black](-6,4)circle(4pt);
			\draw[fill=myblue,color=black](-5,3)circle(4pt);
		\end{tikzpicture}
	\end{minipage}
\caption{For $n,m<0$, Corollary~\ref{cor:wdconv1} (left) and
  Corollary~\ref{cor:wdconv2} (right) translate into convolutions of paths.}
	\label{fig:convolution2}
\end{figure}
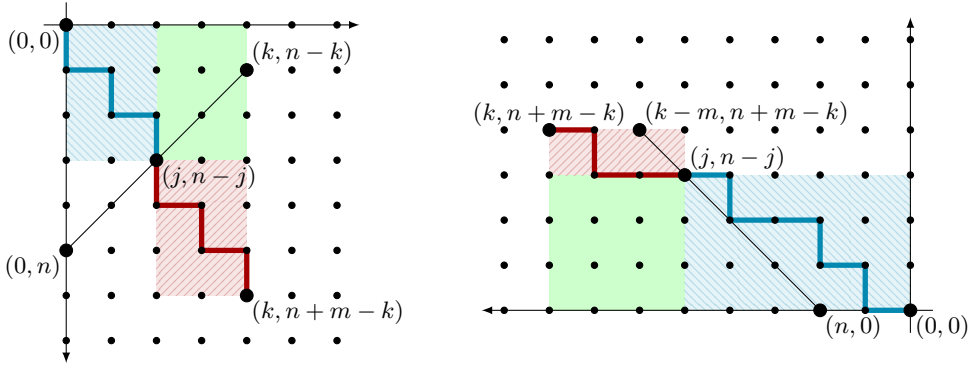
\end{center}
If $n$ and $m$ have mixed signs, the combinatorial interpretation is less
obvious, but we can give a combinatorial proof of the Corollaries by defining
a sign-reversing involution. Since both Corollaries are equivalent, we will
only give a combinatorial proof of Corollary~\ref{cor:wdconv1}. Note that
constructing an involution to prove Corollary~\ref{cor:wdconv2} works very
similar. 

Let $k\geq 0$, $n,m \in \Z$, $\mathrm{sgn}(n) \neq \mathrm{sgn}(m)$. Then
we define $P(n,m,k)$ to be the set of tuples of weighted hybrid lattice
paths $(P_1,P_2)$, where the first path, $P_1$, goes from $(0,0)$ to
$(j,n-j)$ for some $0 \leq j \leq k$ and the second path, $P_2$, goes from
$(j,n-j)$ to $(k,n+m-k)$. (For the second path it is assumed that the origin
is shifted to $(j,n-j)$, such that the path is not restricted to the same
step set as $P_1$.) Recall the weight-assignment to hybrid lattice
paths~\eqref{eqn:area} from Remark~\ref{rem:area}. Let $\mathcal{C}(P)$
be the collection of cells between the $x$-axis and the two paths.
We define the weight of an element $(P_1,P_2) \in P(n,m,k)$ as
\begin{equation*}
  w((P_1,P_2))= (-1)^{S(P_1,P_2)}
  \prod_{(i,j) \in \mathcal{C}(P)} w(i,j)^{\mathrm{sgn}((i,j))}
\end{equation*}
where $\mathrm{sgn}((i,j))=\mathrm{sgn}(i-1)\mathrm{sgn}(j-1)$ and
$S(P_1,P_2)$ is the number of east-south step combinations in either
$P_1$ or $P_2$. See Figures~\ref{fig:involution2} and \ref{fig:involution1}
for examples, where the cells in $\mathcal{C}(P)$ are shaded. By construction,
the sum over all weights of tuples in $P(n,m,k)$ is given by the right hand
side of Corollary~\ref{cor:wdconv1}.
Let $Q(n,m,k)$ be the subset of $P(n,m,k)$ where, after deleting all
intersecting north and south steps of $P_1$ and $P_2$, one of the paths is
completely deleted and we are left with a valid hybrid lattice path from
$(0,0)$ to $(k,n+m-k)$ and an empty path. See Figure~\ref{fig:involution2}
for three examples. The sum over all weights of tuples in $Q(n,m,k)$ is
given by the left hand side of Corollary~\ref{cor:wdconv1}, since in this
case the weighting corresponds to the weighting of hybrid lattice paths
\eqref{eqn:area}. In order to give a combinatorial interpretation of the
Corollary, we will define a sign-reversing involution $\iota$ on $P(n,m,k)$
whose fixed-point set is $Q(n,m,k)$.
\begin{center}
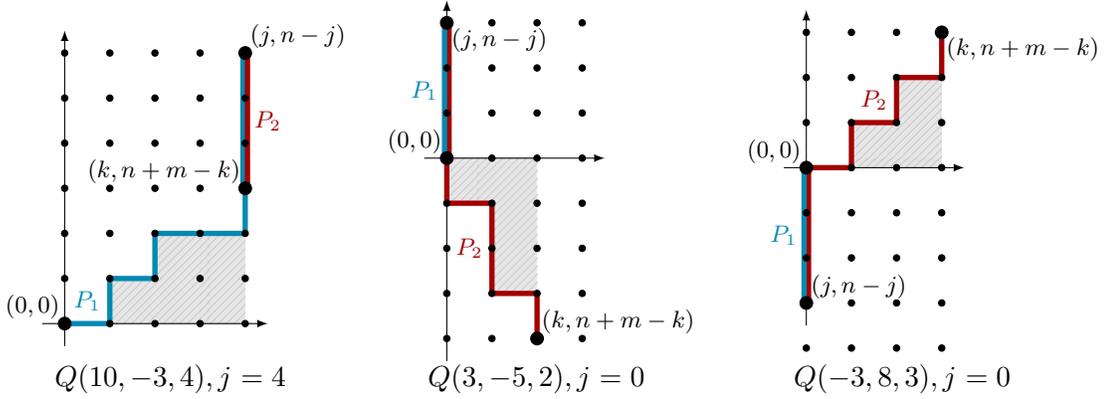
\begin{figure}
\begin{minipage}{0.32\textwidth}
	\centering
	\begin{tikzpicture}[scale=0.6]
		\path[fill=gray!20] (1,0) rectangle (4,1) node[pos=.5, color=black] {};
		\path[pattern=north east lines, pattern color=gray!40] (1,0) rectangle (4,1) node[pos=.5, color=black] {};
		\path[fill=gray!20] (2,1) rectangle (4,2) node[pos=.5, color=black] {};
		\path[pattern=north east lines, pattern color=gray!40] (2,1) rectangle (4,2) node[pos=.5, color=black] {};	
		\draw[->,>=latex] (0,-0.5) -- (0,6.5);
		\draw[->,>=latex] (-0.5,0) -- (4.5,0);
		\draw[color=myblue, line width=2pt] (0,0) -- (1,0) -- (1,1) -- (2,1) -- (2,2) -- (4,2) -- (4,3) -- (3.95,3) -- (3.95,6);
		\draw[color=myred, line width=2pt] (4.05,6) -- (4.05,3);
		\foreach \x in {0,...,4}
		\foreach \y in {0,...,6}
		{\draw[fill=black](\x,\y)circle(2pt);}
		\node[above left=-0.1, color=black] at (0,0) {\footnotesize $(0,0)$};
		\node[above left=-0.1, color=black] at (4,3) {\footnotesize $(k,n+m-k)$};
		\node[above right=-0.1, color=black] at (4,6) {\footnotesize \color{black} $(j,n-j)$};
		\node[color=myblue] at (0.5,0.5) {\footnotesize $P_1$};
		\node[color=myred] at (4.5,4.5) {\footnotesize $P_2$};
		\draw[fill=black,color=black](0,0)circle(4pt);
		\draw[fill=black,color=black](4,3)circle(4pt);
		\draw[fill=myblue,color=black](4,6)circle(4pt);
	\end{tikzpicture}
\end{minipage}
\hfill
\begin{minipage}{0.32\textwidth} \centering
	\begin{tikzpicture}[scale=0.6]
		\path[fill=gray!20] (0,0) rectangle (1,-1) node[pos=.5, color=black] {};
		\path[pattern=north east lines, pattern color=gray!40] (0,0) rectangle (1,-1) node[pos=.5, color=black] {};
		\path[fill=gray!20] (1,0) rectangle (2,-3) node[pos=.5, color=black] {};
		\path[pattern=north east lines, pattern color=gray!40] (1,0) rectangle (2,-3) node[pos=.5, color=black] {};		
		\draw[->,>=latex] (0,-4.5) -- (0,3.5);
		\draw[->,>=latex] (-0.5,0) -- (3.5,0);	
		\draw[color=myred, line width=2pt] (0.05,3) -- (0.05,0) -- (0,0) -- (0,-1) -- (1,-1) -- (1,-3) -- (2,-3) -- (2,-4);
		\draw[color=myblue, line width=2pt] (-0.05,0) -- (-0.05,3);	
		\foreach \x in {0,...,3}
		\foreach \y in {-4,...,3}
		{\draw[fill=black](\x,\y)circle(2pt);}
		\node[above left=-0.1, color=black] at (0,0) {\footnotesize $(0,0)$};
		\node[above right=-0.1, color=black] at (2,-4) {\footnotesize $(k,n+m-k)$};
		\node[below right=-0.1, color=black] at (0,3) {\footnotesize \color{black} $(j,n-j)$};
		\node[color=myblue] at (-0.5,1.5) {\footnotesize $P_1$};
		\node[color=myred] at (0.5,-2) {\footnotesize $P_2$};
		\draw[fill=black,color=black](0,0)circle(4pt);
		\draw[fill=black,color=black](2,-4)circle(4pt);
		\draw[fill=myblue,color=black](0,3)circle(4pt);
	\end{tikzpicture}
\end{minipage}
\hfill
\begin{minipage}{0.32\textwidth}\centering
	\begin{tikzpicture}[scale=0.6]
		\path[fill=gray!20] (1,0) rectangle (2,1) node[pos=.5, color=black] {};
		\path[pattern=north east lines, pattern color=gray!40] (1,0) rectangle (2,1) node[pos=.5, color=black] {};
		\path[fill=gray!20] (2,0) rectangle (3,2) node[pos=.5, color=black] {};
		\path[pattern=north east lines, pattern color=gray!40] (2,0) rectangle (3,2) node[pos=.5, color=black] {};	
		\draw[->,>=latex] (0,-3.5) -- (0,3.5);
		\draw[->,>=latex] (-0.5,0) -- (3.5,0);		
		\draw[color=myred, line width=2pt] (0.05,-3) -- (0.05,0) -- (1,0) -- (1,1) -- (2,1) -- (2,2) -- (3,2) -- (3,3);
		\draw[color=myblue, line width=2pt] (-0.05,0) -- (-0.05,-3);				
		\foreach \x in {0,...,3}
		\foreach \y in {-4,...,3}
		{\draw[fill=black](\x,\y)circle(2pt);}
		\node[above left=-0.1, color=black] at (0,0) {\footnotesize $(0,0)$};
		\node[below right=-0.1, color=black] at (3,3) {\footnotesize $(k,n+m-k)$};
		\node[above right=-0.1, color=black] at (0,-3) {\footnotesize \color{black} $(j,n-j)$};
		\node[color=myblue] at (-0.5,-1.5) {\footnotesize $P_1$};
		\node[color=myred] at (1.5,1.5) {\footnotesize $P_2$};
		\draw[fill=black,color=black](0,0)circle(4pt);
		\draw[fill=black,color=black](3,3)circle(4pt);
		\draw[fill=myblue,color=black](0,-3)circle(4pt);
	\end{tikzpicture}
\end{minipage}
\begin{minipage}{0.32\textwidth}\centering
	$Q(10,-3,4), j=4$
\end{minipage}
\hfill
\begin{minipage}{0.32\textwidth}\centering
	$Q(3,-5,2), j=0$
\end{minipage}
\hfill
\begin{minipage}{0.32\textwidth}\centering
	$Q(-3,8,3), j=0$
\end{minipage}
\caption{Three pairs of paths $(P_1,P_2)$ in the fixed-point set.}
\label{fig:involution2}
\end{figure}
\end{center}
Consider the case $m<0\leq n$. Given an element $(P_1,P_2)$ of $P(n,m,k)$,
$P_1$ ending at $(j,n-j)$, we look at the last east step, $e_1$, of $P_1$ and
the first east step, $e_2$, of $P_2$. If both paths have an east step, we
construct a new pair of paths $(P_1',P_2')$ as follows. If $e_1$ is at the
same height or below $e_2$ (see the right side of Figure~\ref{fig:involution1}
for an example), then $P_1'$ starts by following $P_1$ until $e_1$. $P_1'$ is
following $e_1$ and afterwards it's going north until it reaches $e_2$.
$P_1'$ ends by following $e_2$ and by adding north steps until the path
reaches $(j+1,n-j-1)$. The path $P_2'$ begins at $(j+1,n-j-1)$ with south
steps until it reaches $P_2$ and then follows all remaining steps of $P_2$. 
If $e_1$ lies above $e_2$ (see the left side of Figure~\ref{fig:involution1}
for an example), then $P_1'$ is the path that follows $P_1$ until the starting
point of $e_1$. Instead of following $e_1$, $P_1'$ is going north until it
reaches $(j-1,n-j+1)$. $P_2'$ then starts at $(j-1,n-j+1)$ with south steps
until it reaches $e_1$. $P_2'$ follows $e_1$ and afterwards it follows all
remaining steps of $P_2$ until it reaches $(k,n+m-k)$.
If $P_1$ does not have an east step and $e_2$ is at height $0$ or higher, we
construct $P_1'$ and $P_2'$ as in the case when $e_1$ is weakly below $e_2$
(so $e_2$ will be part of $P_1'$). If $P_2$ does not have an east step and
$e_1$ is at height $n+m-k+1$ or higher, we construct $P_1'$ and $P_2'$ as in
the case when $e_1$ is above $e_2$ (so $e_1$ will be part of $P_2$). In all
other cases, $(P_1,P_2) \in Q(n,m,k)$ and we set $(P_1',P_2')=(P_1,P_2)$. 
Let $\iota$ be the map that maps $(P_1,P_2)$ to $(P_1',P_2')$. By
construction, $(P_1',P_2')$ is an element of $P(n,m,k)$. Also by construction,
$\iota^2=id$, so $\iota$ is an involution. It is not hard to check that
$\iota$ is weight-preserving, while $P_2'$ has one more or one less
east-south step combination than $P_2$ (see Figure~\ref{fig:involution1}
for an example), except if $(P_1,P_2)$ is in $Q(n,m,k)$. This makes $\iota$
sign-reversing outside of the fixed-point set $Q(n,m,k)$. We therefore are
finished with the case $m<0 \leq n$.
\begin{center}
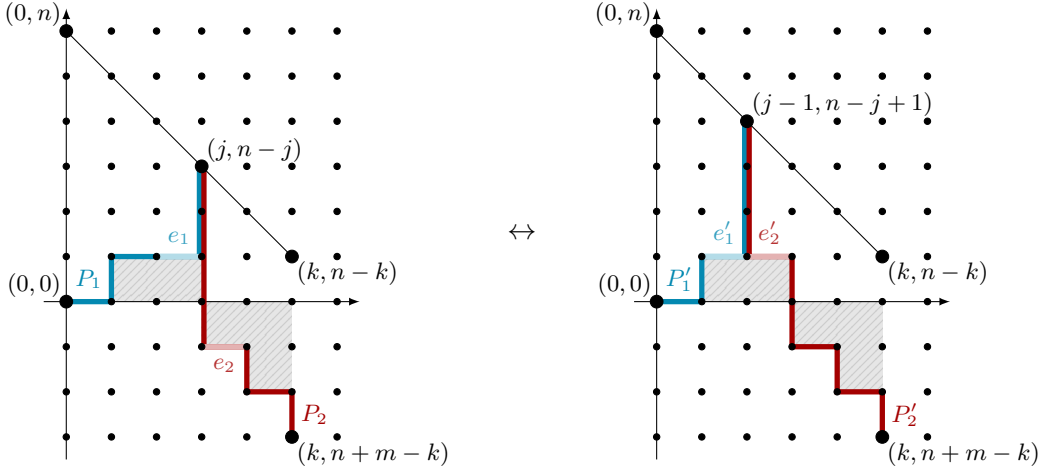
\begin{figure}
\begin{minipage}{0.45\textwidth}
	\centering
	\begin{tikzpicture}[scale=0.6]
		\path[fill=gray!20] (1,0) rectangle (3,1) node[pos=.5, color=black] {};
		\path[pattern=north east lines, pattern color=gray!40] (1,0) rectangle (3,1) node[pos=.5, color=black] {};
		\path[fill=gray!20] (3,0) rectangle (5,-1) node[pos=.5, color=black] {};
		\path[pattern=north east lines, pattern color=gray!40] (3,0) rectangle (5,-1) node[pos=.5, color=black] {};
		\path[fill=gray!20] (4,-1) rectangle (5,-2) node[pos=.5, color=black] {};
		\path[pattern=north east lines, pattern color=gray!40] (4,-1) rectangle (5,-2) node[pos=.5, color=black] {};
		\draw[->,>=latex] (0,-3.5) -- (0,6.5);
		\draw[->,>=latex] (-0.5,0) -- (6.5,0);		
		\draw[color=myblue, line width=2pt] (0,0) -- (1,0) -- (1,1) -- (2,1) -- (2.95,1) -- (2.95,2) -- (2.95,3);
		\draw[color=myred, line width=2pt] (3.05,3) -- (3.05,2) -- (3.05,1) -- (3.05,0) -- (3.05,-1) -- (4,-1) -- (4,-2) -- (5,-2) -- (5,-3);
		\draw[color=myblue!30,line width=2pt] (2,1) -- node[above, color=myblue!80] {\footnotesize $e_1$} (3,1) ;
		\draw[color=myred!30,line width=2pt] (3,-1) -- node[below, color=myred!80] {\footnotesize $e_2$} (4,-1) ;
		\foreach \x in {0,...,6}
		\foreach \y in {-3,...,6}
		{\draw[fill=black](\x,\y)circle(2pt);}
		\node[above left=-0.1, color=black] at (0,0) {\footnotesize $(0,0)$};
		\node[below right=-0.1, color=black] at (5,-3) {\footnotesize $(k,n+m-k)$};
		\node[above left=-0.1, color=black] at (0,6) {\footnotesize $(0,n)$};
		\node[below right=-0.1, color=black] at (5,1) {\footnotesize $(k,n-k)$};
		\node[above right=-0.1, color=black] at (3,3) {\footnotesize \color{black} $(j,n-j)$};
		\node[color=myblue] at (0.5,0.5) {\footnotesize $P_1$};
		\node[color=myred] at (5.5,-2.5) {\footnotesize $P_2$};
		\draw[fill=black,color=black](0,0)circle(4pt);
		\draw[fill=black,color=black](5,-3)circle(4pt);
		\draw[fill=black,color=black](0,6)circle(4pt);
		\draw[fill=black,color=black](5,1)circle(4pt);
		\draw (0,6) -- (5,1);
		\draw[fill=myblue,color=black](3,3)circle(4pt);
	\end{tikzpicture}
\end{minipage}
\hfill
\begin{minipage}{0.03\textwidth} \centering
	$\leftrightarrow$
\end{minipage}
\hfill
\begin{minipage}{0.45\textwidth}\centering
	\begin{tikzpicture}[scale=0.6]
		\path[fill=gray!20] (1,0) rectangle (3,1) node[pos=.5, color=black] {};
		\path[pattern=north east lines, pattern color=gray!40] (1,0) rectangle (3,1) node[pos=.5, color=black] {};
		\path[fill=gray!20] (3,0) rectangle (5,-1) node[pos=.5, color=black] {};
		\path[pattern=north east lines, pattern color=gray!40] (3,0) rectangle (5,-1) node[pos=.5, color=black] {};
		\path[fill=gray!20] (4,-1) rectangle (5,-2) node[pos=.5, color=black] {};
		\path[pattern=north east lines, pattern color=gray!40] (4,-1) rectangle (5,-2) node[pos=.5, color=black] {};	
		\draw[->,>=latex] (0,-3.5) -- (0,6.5);
		\draw[->,>=latex] (-0.5,0) -- (6.5,0);	
		\draw[color=myblue, line width=2pt] (0,0) -- (1,0) -- (1,1) -- (1.95,1) -- (1.95,4);
		\draw[color=myred, line width=2pt] (2.05,4) -- (2.05,1) -- (3,1) -- (3,0) -- (3,-1) -- (4,-1) -- (4,-2) -- (5,-2) -- (5,-3);
		\draw[color=myblue!30,line width=2pt] (1,1) -- node[above, color=myblue!80] {\footnotesize $e_1'$} (2,1) ;
		\draw[color=myred!30,line width=2pt] (2,1) -- node[above, color=myred!80] {\footnotesize $e_2'$} (3,1) ;
		\foreach \x in {0,...,6}
		\foreach \y in {-3,...,6}
		{\draw[fill=black](\x,\y)circle(2pt);}
		\node[above left=-0.1, color=black] at (0,0) {\footnotesize $(0,0)$};
		\node[below right=-0.1, color=black] at (5,-3) {\footnotesize $(k,n+m-k)$};
		\node[above left=-0.1, color=black] at (0,6) {\footnotesize $(0,n)$};
		\node[below right=-0.1, color=black] at (5,1) {\footnotesize $(k,n-k)$};
		\node[above right=-0.1, color=black] at (2,4) {\footnotesize \color{black} $(j-1,n-j+1)$};
		\node[color=myblue] at (0.5,0.5) {\footnotesize $P_1'$};
		\node[color=myred] at (5.5,-2.5) {\footnotesize $P_2'$};
		\draw[fill=black,color=black](0,0)circle(4pt);
		\draw[fill=black,color=black](5,-3)circle(4pt);
		\draw[fill=black,color=black](0,6)circle(4pt);
		\draw[fill=black,color=black](5,1)circle(4pt);
		\draw (0,6) -- (5,1);
		\draw[fill=myblue,color=black](2,4)circle(4pt);
	\end{tikzpicture}
\end{minipage}
\caption{Visualization of the involution on a tuple of paths
  $(P_1,P_2) \in P(6,-4,5)$ and $j=3$. Both tuples of paths have the same
  weight (marked with gray diagonal lines), while $P_2'$ has one more east
  south step combination than $P_2$, so the sign is reversed.}
\label{fig:involution1}	
\end{figure}
\end{center}

The case $n<0\leq m$ works very similar and is therefore left to the reader.

\subsection{Hybrid Sets}\label{sec:hybridsets}
While it is classical that binomial coefficients count the number of lattice
paths, it is also classical that binomial coefficients $\binom{n}{k}$ with
nonnegative integer values count the number of subsets with $k$ elements of
a set with $n$ elements. In \cite{L}, Loeb showed that this interpretation
can be extended to all integers $n,k$ by a generalization of the definition
of sets. Formichella and Straub \cite{FS} introduced a $q$-weighting of these
sets.  After briefly introducing the notation of Loeb (see \cite{FS,L} for
more details and examples), we will show that there is a simple bijection
between hybrid lattice paths and subsets of hybrid new sets that preserves
the $q$-weighting. 

Let $U$ be a collection of elements, then any map $X:U \mapsto \Z$ is called a
\defn{hybrid set}. The value of $X(a)$ is called the \defn{multiplicity} of $a$
in $X$. Hybrid sets are usually denoted in the form $\hSet{\cdots}{\cdots}$,
where elements $a \in U$ with positive multiplicity are listed $X(a)$ times
before the bar and elements with negative multiplicity are listed $-X(a)$
times after the bar. The \defn{cardinality} of a hybrid set,
$|\hSet{\cdots}{\cdots}|$, is the sum of all multiplicities of the set.

For example, given $U=\{1,2,3,4,5,6\}$, the hybrid set
$\hSet{1,1,1,3,3}{2,5,5}$ contains the elements $1,2,3,5$ with multiplicity
$3,-1,2,-2$, respectively, while the elements $4,6$ have multiplicity $0$.
The cardinality of the set is $|\hSet{1,1,1,3,3}{2,5,5}|=3-1+2-2=2$.

A hybrid \defn{new set} is a hybrid set where all multiplicities are either
in $\{0,1\}$ or in $\{0,-1\}$. We define the new set $\nSet{a_l}{a_m}$ by
\begin{equation}\label{def:newset}
\nSet{a_l}{a_m}=\left\{
\begin{array}{ll}
	\hSet{a_{l},a_{l+1},\dots,a_m}{} &m>l-1\\
	\emptyset&m=l-1\\
	\hSet{}{a_{l-1},a_{l-2},\dots,a_{m+1}}&m<l-1
\end{array}\right..
\end{equation}
A hybrid set $Y$ is a \defn{subset} of a hybrid set $X$, if one can repeatedly
decrease the multiplicity of elements in $X$ with nonzero multiplicity to
obtain $Y$ or if one has removed $Y$. 

\begin{example}
  From the hybrid new set $\nSet{a_1}{a_2}=\hSet{a_1,a_2}{}$ one can remove the
  subsets $\emptyset$, $\hSet{a_1}{}$, $\hSet{a_2}{}$ and $\hSet{a_1,a_2}{}$
  to obtain the subsets $\hSet{a_1,a_2}{}$, $\hSet{a_2}{}$, $\hSet{a_1}{}$,
  $\emptyset$. So there are 4 different subsets of $\nSet{a_1}{a_2}$. 
	
  From the hybrid new set $\nSet{a_1}{a_{-2}}=\hSet{}{a_0,a_{-1}}$ one can
  remove three subsets with cardinality $2$, $\hSet{a_0,a_{0}}{}$,
  $\hSet{a_0,a_{-1}}{}$ and $\hSet{a_{-1},a_{-1}}{}$, to obtain three subsets
  with cardinality $-4$, $\hSet{}{a_{0},a_{0},a_{0},a_{-1}}$,
  $\hSet{}{a_{0},a_{0},a_{-1},a_{-1}}$, $\hSet{}{a_{0},a_{-1},a_{-1},a_{-1}}$.
  So there are 3 subsets of $\nSet{a_1}{a_{-2}}$ with cardinality $2$ and $3$
  subsets with cardinality $-4$, while $\nSet{a_1}{a_{-2}}$ has infinitely
  many subsets in total.
\end{example}

Loeb \cite{L} proved that, given a hybrid new set with cardinality $n$, the
number of subsets with cardinality $k$ is equal to the absolute value of
$\binom{n}{k}$ for all $n,k\in\Z$. Formichella and Straub \cite{FS} extended
this result to a $q$-weighted version. In the following, we will establish a
general weighted extension by showing that subsets of hybrid new sets are in
one-to-one correspondence with hybrid lattice paths.

For all $n\in \Z$ let $[n]$ be the hybrid new set 
\begin{equation*}
	[n]=
	\begin{cases} 
		\hSet{1,2,3,\dots,n}{}, & n > 0 \\ 
		\emptyset, & n = 0 \\ 
		\hSet{}{0,-1,-2,\dots,n+1}, & n < 0.
	\end{cases}
\end{equation*}
Let $Y$ be a subset of $[n]$ with $k$ elements, then we write it in a canonical form with ordered elements as follows. If $n\geq 0$, we write $Y=\{y_1,y_2,\dots,y_k| \}$, where $y_{i}\leq y_{i+1}$ and, if $n<0$, we write $Y=\{y_1,y_2,\dots,y_{k}| \}$ (for $k\geq 0$) or $Y=\{| y_1,y_2,\dots,y_{-k} \}$ (for $k<0$) with $y_{i} \geq y_{i+1}$.

We identify $k$-element subsets $Y$ of $[n]$ with hybrid lattice paths ending
at $(k,n-k)$ as follows: If $k\geq 0$, each element $y_i$ of $Y$ corresponds
to an east step ending at $(i,y_i-i)$. If $k<0$, each element $y_i$ of $Y$
corresponds to a west step starting at $(i,y_i-i)$. By connecting the
horizontal steps, the origin and the end-point $(k,n-k)$ with north or south
steps wherever possible, one can check, region by region, that we obtain a
unique hybrid lattice path from $(0,0)$ to $(k,n-k)$, since the lattice path
is determined by its horizontal steps. 
For example, the path in Figure~\ref{fig:path1} corresponds to the subset
$\hSet{1,3,5,6}{}$ of $\hSet{1,2,3,4,5,6}{}$. The left path in
Figure~\ref{fig:path2} corresponds to the subset $\hSet{0,-1}{}$ of
$\hSet{}{0,-1}$ whereas the right path corresponds to the subset
$\hSet{}{0,0,-1,-1}$ of $\hSet{}{0,-1}$.

This correspondence immediately yields
\begin{equation}
	\wBinom{n}{k}{w} = \epsilon \sum_{Y} \prod_{i=1}^{k} W(i,y_i-i),
\end{equation}
where the sum ranges over all $k$-element subsets $Y=\nSet{y_1}{y_k}$ (in
canonical form) of $[n]$ and $\epsilon= \mathrm{sgn}(k)^n \mathrm{sgn}(n)^k$.
For $w(s,t)=q$ (and $W(s,t)=q^{t}$) this reduces to the weighted model of
Formichella and Straub \cite[Theorem 4.5]{FS}, which can be stated as
\begin{equation*}
	\eBinom{n}{k}{q} = \epsilon \sum_{Y} q^{\sigma(Y)-k(k+1)/2},
\end{equation*}
where the sum ranges over all $k$-element subsets $Y=\nSet{y_1}{y_k}$ of
$[n]$ and $\sigma(Y)=\sum_{y_i \in Y} y_i$.

\section{Symmetric Functions}
In \cite{Schl1} two important specializations of the weights $w(s,t)$
involving symmetric functions were worked out in the case $n,k \geq 0$. Here
we discuss how this approach generalizes naturally to the case $n,k \in \Z$. 

In \cite{DAL} Damiani, D'Antona and Loeb define generalized complete
homogeneous symmetric functions and elementary symmetric functions over
hybrid sets. For our purposes it is sufficient to look at the following
simplified definitions. See \cite{DAL} for a more extensive discussion. 

Recall the definition of the hybrid set $\nSet{a_1}{a_n}$ from
\eqref{def:newset} and define the \defn{elementary symmetric function}
$e_k(\nSet{a_1}{a_n})$ and the \defn{complete homogeneous symmetric function}
$h_k(\nSet{a_1}{a_n})$ for $n \in \Z$ and $k\geq 0$ by
\begin{subequations}\label{def:sympos}
\begin{align} 
  \prod_{i=1}^{n} (1+a_i t) &= \sum_{k \geq 0} e_k(\nSet{a_1}{a_n}) t^k\\
\intertext{and}
\prod_{i=1}^{n} (1-a_i t)^{-1} &= \sum_{k \geq 0} h_k(\nSet{a_1}{a_n}) t^k.
\end{align}
\end{subequations}
In the spirit of Theorem~\ref{thm:extbinomial} it makes sense to extend these
definitions for all integers $n$ and $k \leq n$ or $k \leq -n$ by
\begin{subequations} \label{def:symneg}
  \begin{align}
    \prod_{i=1}^{n} (1+a_i t)&= \sum_{k \leq n} e_k(\nSet{a_1}{a_n}) t^k\\
                             \intertext{and}
\prod_{i=1}^{n} (1-a_i t)^{-1}&= \sum_{k \leq -n} h_k(\nSet{a_1}{a_n}) t^k,
  \end{align}
  \end{subequations}
  where we expand the left side in the variable $t^{-1}$. For example,
  for $n=-1$ we get:
\begin{align*}
  \prod_{i=1}^{-1} (1+a_i t)& = \frac{1}{1+a_0 t}
  = a_0^{-1}t^{-1}  \frac{1}{1+a_0^{-1} t^{-1}}\\
&= a_0^{-1}t^{-1} \sum_{k \geq 0} (-1)^k a_0^{-k} t^{-k}
 = \sum_{k \leq -1} (-1)^k a_0^{k} t^{k}.
\end{align*}
For convenience, we will also use the abbreviated notations $e_k(n)$
and $h_k(n)$, for $e_k(\nSet{a_1}{a_n})$ and $h_k(\nSet{a_1}{a_n})$,
respectively. 
The elementary and complete homogeneous symmetric functions satisfy for all
$n,k \in \Z$, provided that $(n+1,k)\neq(0,0)$, the recurrence relations 
\begin{align*}
	e_k(n+1) &= e_k(n)+a_{n+1}e_{k-1}(n) \\
	h_k(n) &= h_k(n-1)+a_{n}h_{k-1}(n)	
\end{align*}
with initial conditions
\begin{align*}
	e_0(n)=1, & \quad h_0(n) = 1 \\
	e_n(n)=\prod_{i=1}^{n} a_i, & \quad h_n(1) = \mathrm{sgn}(n) a_1^n.
\end{align*}

We can deduce directly from the definitions that for all $n,k \in \Z$
there holds (see also \cite[Proposition 1]{DAL})
\begin{subequations} \label{eq:symsym}
\begin{align}
  e_k(\nSet{a_1}{a_n})
  &=h_k(\nSet{(-a_{n+1})}{(-a_{0})})\label{eq:symsym1}\\
  e_k(\nSet{a_1}{a_n})
  &=e_{n-k}(\nSet{a_1^{-1}}{a_n^{-1}}) \prod_{i=1}^{n} a_i. \\
  h_k(\nSet{a_1}{a_n})
  &=h_{-n-k}(\nSet{a_1^{-1}}{a_n^{-1}}) (-1)^n \prod_{i=1}^{n} a_i^{-1} 		
\end{align}
\end{subequations}

\begin{remark}
  Loeb \cite{L} defined ordinary binomial coefficients for $n \in \Z$ and
  $k\geq 0$ as $\binom{n}{k}=e_{k}(\nSet{a_1}{a_n})$ with $a_i=1$ for all
  $i \in \Z$. He extended this definition for negative $k$ by an explicit
  formula containing gamma functions. With the extension \eqref{def:symneg}
  of the definition of $e_k(n)$ and $a_i=1$ for all $i\in \Z$ we now have
  the equality $\binom{n}{k}=e_{k}(\nSet{a_1}{a_n})$ which holds for all
  $n,k \in \Z$. 
\end{remark}

For nonnegative integers $n$ and $k$ the definitions \eqref{def:sympos}
correspond to the sum formulae
\begin{align}
  e_k(\{a_1,a_2,\dots,a_n\})
  &= \sum_{1\leq j_1 < j_2 < \dots < j_k \leq n} a_{j_1} a_{j_2}
    \cdots a_{j_k} \label{def:epos} \\
  h_k(\{a_1,a_2,\dots,a_n\})
  &= \sum_{1\leq j_1 \leq j_2 \leq \dots \leq j_k \leq n} a_{j_1} a_{j_2}
    \cdots a_{j_k} \label{def:hpos}
\end{align}
and $e_0(n)=h_0(n)=1$ (see for example \cite[Proposition 4.3.5]{Sa}).

This definition extends even to the case where $n$ might be negative as
follows. Analogously to the extended definition for products in \eqref{prod},
we define sums generally by

\begin{equation}\label{sum}
\sum_{j=l}^mA_j=\left\{
\begin{array}{ll}
	A_l+A_{l+1}+\cdots +A_m&m>l-1\\
	0&m=l-1\\
	-A_{l-1}-A_{l-2}-\cdots -A_{m+1}&m<l-1
\end{array}\right..
\end{equation}

Then, for $n<0$ and $k>0$, the sum in \eqref{def:epos} becomes
\begin{align*}
 \sum_{1\leq j_1 < j_2 < \dots < j_k \leq n} a_{j_1} a_{j_2} \cdots a_{j_k}
  &=\sum_{j_1 = 1}^{n-k+1} \sum_{j_2 = j_1 +1}^{n-k+2}
    \cdots \sum_{j_k = j_{k-1}+1}^{n} a_{j_1} a_{j_2} \cdots a_{j_k} \\
  &=\sum_{j_1 = n-k+2}^{0} \sum_{j_2 = n-k+3}^{j_1}
    \cdots \sum_{j_k = n+1}^{j_{k-1}} (-1)^k a_{j_1} a_{j_2} \cdots a_{j_k} \\
  &=\sum_{n+1 \leq j_k \leq j_{k-1} \leq \dots \leq j_{1} \leq 0}
    (-a_{j_1}) (-a_{j_2}) \cdots (-a_{j_k})
\end{align*}
which is by \eqref{def:hpos} equal to $h_k(\{-a_{n+1},-a_{n+2},\dots,-a_0\})$.
Indeed, also by \eqref{eq:symsym} we have
$e_k(\hSet{}{a_0,a_{-1},\dots, a_{n+1}}=
h_k(\hSet{-a_{n+1},-a_{n+2},\dots,-a_0}{})$.

\subsection{Elementary symmetric functions}
The first choice of the weights is $w(s,t)=\frac{a_{s+t}}{a_{s+t-1}}$.
In this case we have $\wBinomText{n}{k}{w} = e_k(n)\prod_{i=1}^{k} a_i^{-1}$.
The noncommutative binomial theorem in Theorem~\ref{thm:binomial}
now gives the expansions
\begin{subequations}\label{eqn:e_bin}
\begin{align}
  (x+y)^n &= \sum_{k\geq0} e_k(n)\Big(\prod_{i=1}^{k} a_i\Big) x^k y^{n-k}\\
  \intertext{and}
  (x+y)^n& = \sum_{k\leq n} e_k(n)\Big(\prod_{i=1}^{k} a_i\Big) x^{k} y^{n-k}
\end{align}
\end{subequations}
for $x$ and $y$ noncommutative invertible variables satisfying
\begin{subequations}\label{eqn:e_noncommrel3}
\begin{align}
yx &= \frac{a_2}{a_1} xy,\\
x \frac{a_{s+t}}{a_{s+t-1}}&= \frac{a_{s+t+1}}{a_{s+t}} x,\\
y \frac{a_{s+t}}{a_{s+t-1}}&= \frac{a_{s+t+1}}{a_{s+t}} y.
\end{align}
\end{subequations}
\begin{example}
  Define $\epsilon_a$ to be an operator that shifts all $a_i$ to $a_{i+1}$,
  $i\in\Z$. For example,
  $\epsilon_a (a_1 + \frac{a_{-2}a_{3}}{a_5}) = a_2 + \frac{a_{-1}a_{4}}{a_6}$.
  Now let $x=a_1 t \epsilon_a$ and $y=\epsilon_a$, such that all relations in
  \eqref{eqn:e_noncommrel3} are satisfied. Evaluating \eqref{eqn:e_bin}
  by applying all operators on both sides recovers the generating functions
\begin{align*}
  \prod_{i=1}^{n} (1+a_i t)
  &= \sum_{k \geq 0} e_k(n) t^k \quad \text{and} \quad 
\prod_{i=1}^{n} (1+a_i t) = \sum_{k \leq n} e_k(n) t^k\\
 \intertext{and applying the duality \eqref{eq:symsym}
 and $h_k(\nSet{a_1}{a_n})(-1)^k = h_k(\nSet{(-a_1)}{(-a_n)})$ yields} 
 \prod_{i=1}^{n} (1-a_i t)^{-1} &= \sum_{k \geq 0} h_k(n) t^k \quad \text{and}
\quad \prod_{i=1}^{n} (1-a_i t)^{-1}  = \sum_{k \leq -n} h_k(n) t^k,
\end{align*}

A second choice of $x$ and $y$ would be $x=-a_1 \epsilon_a$ and
$y=t \epsilon_a$. Again, all relations in \eqref{eqn:e_noncommrel3} are
satisfied and we obtain another well-known generating function
(see \cite[p.~208]{DAL})
\begin{align}
  (t-a_1)(t-a_2)\cdots(t-a_n)
  &=\sum_{k \geq 0} e_k(\nSet{(-a_1)}{(-a_n)}) t^{n-k} \label{eqn:e_gf}\\
\intertext{for $n\geq 0$ and}
  (t-a_1)^{-1}(t-a_2)^{-1}\cdots(t-a_n)^{-1}
  &=\sum_{k \geq 0} e_k(\nSet{(-a_1)}{(-a_n)}) t^{n-k}\nonumber\\
    \intertext{and}
  (t-a_1)^{-1}(t-a_2)^{-1}\cdots(t-a_n)^{-1}
  &=\sum_{k \leq n} e_k(\nSet{(-a_1)}{(-a_n)}) t^{n-k}\nonumber
\end{align}
for $n<0$.
\end{example}

The convolutions in Corollary~\ref{cor:wdconv1} and \ref{cor:wdconv2},
respectively, give
\begin{align}\label{eqn:econv1id}
 	e_k(n+m)&=\sum_{j=0}^{k}
 	e_j(\nSet{a_1}{a_n})
 	e_{k-j}(\{a_{n+1}\cdots a_{n+m}\})\\
\intertext{and}
\label{eqn:econv2id}
 	e_k(n+m)&=\sum_{j=k-m}^{n}
 	e_j(\nSet{a_1}{a_n})
 	e_{k-j}(\nSet{a_{n+1}}{a_{n+m}}).
\end{align}
Note that these convolutions are extensions of the corresponding convolution
in \cite[(3.6a)]{Schl1}, since $n, m$ and $k$ can be negative.
One consequence of this observation is the following example.
\begin{example}
  Choose $m=-n$ in \eqref{eqn:econv1id}. Then, by using the duality of $e_k$
  and $h_k$ \eqref{eq:symsym}, we recover the well-known identity
  (cf. \cite{Sa})
\[
  \sum_{j=0}^{k} (-1)^j e_j(\nSet{a_1}{a_n}) h_{k-j}(\nSet{a_1}{a_n})
  = \delta_{k,0}.
\]
More generally, we obtain
\[
  \sum_{j=0}^{k} e_j(\nSet{a_1}{a_n}) h_{k-j}(\nSet{(-a_{n+m+1})}{(-a_{n})})
  = e_k(n+m).
\]
\end{example}

\subsection{Complete homogeneous symmetric functions}
The second choice of the weights is $w(s,t)=\frac{a_{t+1}}{a_{t}}$. In this
case we have $\wBinomText{n}{k}{w} = \mathrm{sgn}(k) h_k(n-k+1) a_1^{-k}$.
The noncommutative binomial theorem in Theorem~\ref{thm:binomial} now gives
the expansions
\begin{subequations}\label{eqn:h_bin}
\begin{align}
  (x+y)^n &= \sum_{k\geq0} h_k(n-k+1)a_{1}^{-k} x^k y^{n-k}\\
  \intertext{and}
  \quad (x+y)^n
          &= \sum_{k\leq n} \mathrm{sgn}(k) h_k(n-k+1)a_{1}^{-k} x^{k} y^{n-k}
\end{align}
\end{subequations}
for $x$ and $y$ noncommutative invertible variables satisfying
\begin{subequations}\label{eqn:h_noncommrel3}
	\begin{align}
		yx &= \frac{a_2}{a_1} xy,\\
		x \frac{a_{t+1}}{a_{t}}&= \frac{a_{t+1}}{a_{t}} x,\\
		y \frac{a_{t+1}}{a_{t}}&= \frac{a_{t+2}}{a_{t+1}} y.
	\end{align}
\end{subequations}
\begin{example}
 Let $x=a_1$ and $y=(t-a_1)\epsilon_a$, such that all relations in
 \eqref{eqn:h_noncommrel3} are satisfied. Evaluating \eqref{eqn:h_bin}
 by applying all operators on both sides recovers the generating functions 
\begin{align}
  t^n&=\sum_{k \geq 0} h_k(\nSet{a_1}{a_{n-k+1}})
       (t-a_1)(t-a_2)\cdots (t-a_{n-k}) \label{eqn:h_gf}\\
\intertext{for $n\geq 0$ (see also \cite[p. 208]{DAL}) and}
  t^n&=\sum_{k \geq 0} h_k(\nSet{a_1}{a_{n-k+1}})
       (t-a_1)^{-1}(t-a_2)^{-1}\cdots (t-a_{n-k})^{-1} \nonumber\\
  \intertext{and}
  t^n&=\sum_{k \leq n} \mathrm{sgn}(k) h_k(\nSet{a_1}{a_{n-k+1}})
       (t-a_1)(t-a_2)\cdots (t-a_{n-k})\nonumber
\end{align}
for $n<0$.
\end{example}

The convolutions in Corollary~\ref{cor:wdconv1} and \ref{cor:wdconv2},
respectively, give
\begin{align}\label{eqn:hconv1id}
 &h_k(n+m-k+1)\nonumber\\&=\sum_{j=0}^{k}
 	h_j(\nSet{a_1}{a_{n-j+1}})
 	h_{k-j}(\nSet{a_{n-j+1}}{a_{n+m-k+1}})\\
 	\intertext{and}
  &h_k(n+m-k+1)\mathrm{sgn}(k)\nonumber\\
  &=\sum_{j=k-m}^{n} \mathrm{sgn}(j) \mathrm{sgn}(k-j)
    h_j(\nSet{a_1}{a_{n-j+1}})
 	h_{k-j}(\nSet{a_{n-j+1}}{a_{n+m-k+1}}).\label{eqn:hconv2id}
\end{align}

\begin{remark}
	By applying the involution~\eqref{eqn:inv6} to the weight $w(s,t) = \frac{a_{s+t}}{a_{s+t-1}}$, we obtain
	\[
		\breve{w}(s,t)=w(s,1-s-t)^{-1}=\frac{a_{-t}}{a_{1-t}},
	\]
	which is the weight for the complete homogeneous symmetric functions subject to the substitution $a_n \mapsto a_{-n+1}$. The reflection formula~\eqref{eqn:refl2b} reduces to~\eqref{eq:symsym1} in this case.
\end{remark}
\subsection{Application to Stirling numbers}

Here we define weighted integers 
\begin{equation}
{}_w [n] := \wBinom{n}{1}{w}.\label{eqn:weighted_n}
\end{equation}
Note that by Proposition \ref{prop:six}, ${}_w[n]$ is well-defined for all
$n\in \mathbb{Z}$.

It is well known that the Stirling numbers of the first kind and the
second kind, denoted by $s(n,k)$ and $S(n,k)$, respectively,
can be defined as the connecting coefficients of the following identities 
\begin{subequations}\label{eqn:Stirling_gf}
\begin{align}
 &t (t-1)\cdots (t-n+1)=\sum_{k=0}^n s (n,k) t^k\label{eqn:Stirling_1}\\
\intertext{and}
&t^n = \sum_{k=0}^n S(n,k) t(t-1)\cdots (t-k+1).\label{eqn:Stirling_2}
\end{align}
\end{subequations}
We also have the recursions 
$$s(n,k)=s(n-1, k-1)-(n-1) s (n-1,k)$$
and 
$$S(n,k)=k S(n-1, k)+S(n-1, k-1)$$
which can uniquely determine the Stirling numbers, given the initial
conditions $s(n,0)=\delta_{n,0}$ and $s(n,k)=0$ for $k>n$, for the Stirling
numbers of the first kind, and $S(n,0)=\delta_{n,0}$ and $S(n,k)=0$ for $k>n$,
for the Stirling numbers of the second kind.
Comparing \eqref{eqn:Stirling_1} to the generating function of the
elementary symmetric function given in \eqref{eqn:e_gf} gives the well known
fact that $s(n,k)$ is the $(n-k)^{\text{th}}$ elementary symmetric function
of $-1, -2,\dots, -n+1$ (cf. \cite[I.2, Ex. 11 (a)]{Mac}).
Similarly, comparison of \eqref{eqn:Stirling_2} and  \eqref{eqn:h_gf}
gives the fact that $S(n,k)$ is the $(n-k)^{\text{th}}$ 
complete homogeneous symmetric function of $1, 2,\dots, k$
(cf. \cite[I.2, Ex. 11 (b)]{Mac}).

We replace each integer showing up in \eqref{eqn:Stirling_gf} by the
weighted integer defined in \eqref{eqn:weighted_n} to 
define weighted analogues of the Stirling numbers 
\begin{subequations}\label{eqn:wStirling_gf}
\begin{align}
  &t (t-{}_w[1])\cdots (t-{}_w[n-1])
    =\sum_{k=0}^n s_w (n,k) t^k\label{eqn:wStirling_1}\\ 
\intertext{and}
  & t^n = \sum_{k=0}^n S_w (n,k) t(t-{}_w[1])
    \cdots (t-{}_w[k-1]).\label{eqn:wStirling_2}
\end{align}
\end{subequations}
From \eqref{eqn:wStirling_gf}, we can obtain the recurrence relations 
\begin{subequations}\label{eqn:wStirling_rec}
\begin{align}
s_w (n,k)&=s_w (n-1, k-1)-{}_w[n-1] s_w (n-1,k)\label{eqn:wStirling_1_rec}\\
\intertext{and}
S_w(n,k)&={}_w[k] S_w(n-1, k)+S_w(n-1, k-1).\label{eqn:wStirling_2_rec}
\end{align}
\end{subequations}
By comparing \eqref{eqn:wStirling_1} to the generating function of the
elementary symmetric function given in \eqref{eqn:e_gf}, 
for any nonnegative integers $n$ and $k$, we get 
\begin{equation}\label{eqn:Stirling1_e}
s_w (n,k)=e_{n-k}(\{-{}_w[0], -{}_w[1], -{}_w[2],\dots, -{}_w[n-1]\}).
\end{equation}
Furthermore, the recurrence relation \eqref{eqn:wStirling_1_rec} is valid
for negative $n$ and $k$ values with $k\le n$.
As a result of using the recurrence relation for $k\le n<0$, we can prove
that \eqref{eqn:Stirling1_e} extends to hold 
for $n, k \in \mathbb{Z}$ with $k\le n$.

Similarly, for the weighted analogue of the Stirling numbers of the
second kind, we compare \eqref{eqn:wStirling_2} to \eqref{eqn:h_gf} and
obtain, for any nonnegative integers $n$ and $k$, 
\begin{equation}\label{eqn:Stirling2_h}
S_w (n,k)=h_{n-k}(\{ {}_w[0], {}_w[1],\dots, {}_w[k]\}).
\end{equation}
Also, the recurrence relation \eqref{eqn:wStirling_2_rec} allows us to
extend the definition of $S_w (n,k)$ for negative $n$ and $k$ values,
with $k\le n$. As a result,  we can prove that \eqref{eqn:Stirling2_h}
holds for $k\le n<0$ as well.


Identities \eqref{eqn:econv1id}, \eqref{eqn:econv2id}, \eqref{eqn:hconv1id}
and \eqref{eqn:hconv2id} satisfied by elementary symmetric functions and
complete homogeneous symmetric functions coming from the convolution
formulas give relations satisfied by weighted analogues of the Stirling
numbers. To do that, however, we need a slightly more generalized definition
of the weighted Stirling numbers.
Namely, we define \emph{$\alpha$-shifted weighted Stirling numbers} by 
\begin{subequations}\label{eqn:wStirling_gf_a}
\begin{align}
  &(t-{}_w[\alpha]) (t-{}_w[\alpha+1])\cdots (t-{}_w[\alpha+n-1])
    =\sum_{k=0}^n s_w ^\alpha (n,k) t^k\label{eqn:wshiftedStirling_1}\\
\intertext{and}
  &t^n = \sum_{k=0}^n S_w ^\alpha(n,k)( t-{}_w[\alpha])(t-{}_w[\alpha+1])
    \cdots (t-{}_w[\alpha+ k-1]),\label{eqn:wshiftedStirling_2}
\end{align}
\end{subequations}
for some parameter $\alpha$. Note that we recover the weighted Stirling
numbers when $\alpha=0$. Then again by comparing to the generating functions
of symmetric functions, we get
\begin{subequations}\label{eqn:walSt-e-h}
\begin{align}
  s_w ^\alpha(n,k)
  &=
e_{n-k}(\{-{}_w[\alpha], -{}_w[\alpha+1],\dots, -{}_w[\alpha + n-1]\}),\\
\intertext{and} 
  S_w ^\alpha(n,k)
 &=h_{n-k}(\{{}_w[\alpha], {}_w[\alpha+1],\dots, {}_w[\alpha+k]\}),
\end{align}
\end{subequations}
for any $n, k\in \mathbb{Z}$ with $k\le n$.
These two identities, combined with \eqref{eq:symsym1}, immediately
imply the relation
\begin{equation}\label{eqn:dual_wal}
  s_w ^\alpha(n,k)=  S_w ^{\alpha+n}(-k-1,-n-1),
\end{equation}
valid for $n, k\in \mathbb{Z}$ with $k\le n$,
that connects the first and second kinds of $\alpha$-shifted
weighted Stirling numbers. This is actually an extension of
a well-known duality for the Stirling numbers which takes the form
\begin{equation}\label{eqn:dual}
s(n,k)=(-1)^{n-k}S(-k,-n),
\end{equation}
that can be recovered from \eqref{eqn:dual_wal} by letting $\alpha=0$
and the letting the weights $w(s,t)\to 1$.
The duality relation in \eqref{eqn:dual} is very classical
and appears, for instance, in the (almost) two centuries old
treatise \cite[p.~305]{vE} by von Ettingshausen, in different notation.
See also the discussion by Knuth~\cite{Kn}.

The recurrence relations for the $\alpha$-shifted weighted Stirling numbers
that extend those in \eqref{eqn:wStirling_rec} are
\begin{subequations}\label{eqn:walStirling_rec}
\begin{align}
  s_w^\alpha (n,k)&=s_w ^\alpha(n-1, k-1)-{}_w[\alpha+n-1]
                    s_w^\alpha (n-1,k)\label{eqn:walStirling_1_rec}\\
\intertext{and}
  S_w^\alpha(n,k)&={}_w[\alpha+k] S_w^\alpha(n-1, k)
                   +S_w^\alpha(n-1, k-1).\label{eqn:walStirling_2_rec}
\end{align}
\end{subequations}
The identities \eqref{eqn:econv1id},
\eqref{eqn:econv2id}, \eqref{eqn:hconv1id} and \eqref{eqn:hconv2id} can be 
written in terms of weighted (and shifted weighted) Stirling numbers
as follows:
\begin{align*}
S_w (n+m, n+m-k) &= \sum_{j=0}^k S_w (n, n-j) s_w ^{n+m-k+1}(k-j-m-1, -m-1),\\
S_w (n+m, n+m-k) &= \sum_{j=k-m}^n S_w (n, n-j) s_w ^{n+m-k+1}(k-j-m-1, -m-1),
\end{align*}
and 
\begin{align*}
s_w (n+m, k-n-m) &= \sum_{j=0}^k s_w (n, n-j) S_w ^{n+m} (k-j-m-1, -m-1),\\
s_w (n+m, k-n-m) &= \sum_{j=k-m}^n s_w (n, n-j) S_w ^{n+m} (k-j-m-1, -m-1).
\end{align*}

\begin{remark}
  The $\alpha$-shifted Stirling numbers have been defined by Remmel and Wachs
  in \cite{RW} when the weighted integer ${}_w [n] $ is given by $[n]_{p,q}$ 
where
$$[n]_{p,q}=\frac{p^n -q^n}{p-q}.$$
\end{remark}

\section{Elliptic Hypergeometric Series}
A very important specialization of the weights in \cite{Schl1}, which also
served as a major motivation for the present work, is the ``elliptic''
specialization. In this section we will recall some important notions of
the theory of elliptic hypergeometric series and study the results from
the previous sections in the elliptic case.

Crucial in the theory of elliptic functions is the
\defn{modified Jacobi theta function}, defined by
\[
	\theta(x;p) \eqdef \prod_{j\geq 0} 
	\left(
	(1-p^jx)(1-\frac{p^{j+1}}{x})
	\right),
	\quad
	\theta(x_1,\dots, x_\ell;p) = 
	\prod_{k=1}^\ell \theta(x_k;p),
\]
where $x,x_1,\dots,x_\ell \neq 0$ and $|p|<1$.
The modified Jacobi theta function satisfies some fundamental
relations \cite[cf. p.~451, Example~5]{Web}, like the inversion formula
\begin{subequations}
	\begin{align}\label{eqn:ellinversion}
		\theta(x;p)=-x \theta(1/x;p),
	\end{align}
	the quasi-periodicity relation
	\begin{align}
		\theta(px;p)=-\frac{1}{x} \theta(x;p),
	\end{align}
	and the three-term identity
	\begin{align}\label{eqn:threeterm}
		\theta(xy,x/y,uz,u/z;p)=
		\theta(uy,u/y,xz,x/z;p)+
		\frac{x}{z}
		\theta(zy,z/y,ux,u/x;p).
	\end{align}
\end{subequations}
We define the theta shifted factorial as
\begin{equation}\label{def:tsf}
	(x;q,p)_k~=~\prod_{i=0}^{k-1}~\theta(xq^i;p),
\end{equation}
where the product is defined for all integers $k$ by \eqref{prod}
and for brevity, we write
\[
(x_1,x_2,\dots,x_\ell;q,p)_k=(x_1;q,p)_k(x_2;q,p)_k \cdots (x_\ell;q,p)_k.
\]
An \defn{elliptic function} is a function of a complex variable that is
meromorphic and doubly periodic. It is well known that elliptic functions can
be obtained as quotients of modified Jacobi theta functions \cite{Ros,Web}.

Let $q=e^{2\pi i\sigma}$, $p=e^{2\pi i\tau}$, with complex $\sigma$, $\tau$,
then a multivariate function over $\C$ of the complex variables
$u_1,\dots,u_n$ is called \defn{totally elliptic}, if it is meromorphic in
each variable with equal periods, $\sigma^{-1}$ and $\tau\sigma^{-1}$, of
double periodicity. The \defn{field of totally elliptic multivariate
functions} is denoted by $\E_{q^{u_1},\dots,q^{u_n};q,p}$.

\subsection{Elliptic weights}
We consider the elliptic specialization of the weights $w(s,t)$
from \cite{Schl1}, here defined for all $s,t \in \Z$ by
\begin{equation}\label{def:ellweights}
	w_{a,b;q,p}(s,t)\eqdef \frac{\ta(aq^{s+2t},bq^{2s+t-2},aq^{t-s-1}/b)}
{\ta(aq^{s+2t-2},bq^{2s+t},aq^{t-s+1}/b)}q
\end{equation}
to obtain the big weights
\begin{equation}\label{def:ellbigweights}
  W_{a,b;q,p}(s,t)\eqdef
  \frac{\ta(aq^{s+2t},bq^{2s},bq^{2s-1},aq^{1-s}/b,aq^{-s}/b)}
{\ta(aq^s,bq^{2s+t},bq^{2s+t-1},aq^{1+t-s}/b,aq^{t-s}/b)}q^t
\end{equation}
for all $s,t \in \Z$ and the \defn{elliptic binomial coefficients}
\begin{equation}\label{def:ellcoef}
  \eBinom{n}{k}{a,b;q,p} \eqdef
  \frac{(q^{1+k},aq^{1+k},bq^{1+k},aq^{1-k}/b;q,p)_{n-k}}
{(q,aq,bq^{1+2k},aq/b;q,p)_{n-k}}.
\end{equation}

It is not hard to check that the elliptic binomial coefficient satisfies
\begin{align*}
&\eBinom{n}{0}{a,b;q,p}=\eBinom{n}{n}{a,b;q,p} =1
\qquad\text{for\/ $n\in \mathbb{Z}$ },\\
\intertext{and for $n,k\in\mathbb{Z}$, provided that $(n+1,k)\neq(0,0)$,}
		&\eBinom{n+1}{k}{a,b;q,p}=
		\eBinom{n}{k}{a,b;q,p}
		+\eBinom{n}{k-1}{a,b;q,p}
		\,W_{a,b;q,p}(k,n+1-k),
\end{align*}
where the recurrence relation is a consequence of the three-term
relation \eqref{eqn:threeterm}. The weight functions \eqref{def:ellweights}
and \eqref{def:ellbigweights} and the elliptic binomial coefficient
\eqref{def:ellcoef} are totally elliptic, i.e., periodic in each of
$\log_q a$, $log_q b$, $k$ and $n$, with equal periods of double periodicity
\cite{Schl0,Schl1}. We denote the field of totally elliptic functions over
$\C$, in the complex variables $\log_qa$ and $\log_qb$, with equal periods
$\sigma^{-1}$, $\tau\sigma^{-1}$ (where $q=e^{2\pi i\sigma}$, $p=e^{2\pi i\tau}$
$\sigma,\tau\in\C$), of double periodicity, by $\E_{a,b;q,p}$.

We now focus on an elliptic version of the noncommutative algebra from
Definition~\ref{def:Cwxy3}. The following definition is not a direct
specialization of the weight-dependent algebra but a convenient extension
of the specialization. 
\begin{definition}\label{def:Cwxy3ell}
Let $x,x^{-1},y,y^{-1},a,b$ be noncommuting variables such that $a$ and $b$
commute with each other. Further let $q$ and $p$ be two complex numbers
with $|p|<1$. Let $\C_{a,b;q,p}[x,x^{-1},y,y^{-1}]$ denote the associative
unital algebra over $\C$, generated by the invertible variables
$x,x^{-1}, y,y^{-1}$ and the set of all totally elliptic functions
$\mathbb E_{a,b;q,p}$, satisfying the relations
\begin{subequations}\label{eqn:noncommrel3ell}
\begin{align}
x^{-1}x &=xx^{-1} = 1,\label{subeqn:rel1ell}\\
y^{-1}y &=yy^{-1} = 1, \label{subeqn:rel2ell}\\
yx &= w_{a,b;q,p}(1,1)xy,\label{subeqn:rel3ell}\\
x f(a,b)&= f(aq,bq^2)x,\label{subeqn:rel4ell}\\
y f(a,b)&= f(aq^2,bq)y\label{subeqn:rel5ell},
\end{align}
\end{subequations}
for all $f\in\E_{a,b;q,p}$.
\end{definition} 
We refer to the variables $x,x^{-1},y,y^{-1},a,b$ forming
$\C_{a,b;q,p}[x,x^{-1},y,y^{-1}]$ as {\em elliptic-commuting} variables.
The actions of $x$ and $y$ on a weight $w(s,t)$ in \eqref{subeqn:rel4} and
\eqref{subeqn:rel5} correspond to the shifts of the parameters $a$ and $b$
as described in \eqref{subeqn:rel4ell} and \eqref{subeqn:rel5ell}.

The elliptic binomial coefficients appeared first in \cite{Schl0} to obtain
an elliptic weighted enumeration of lattice paths and nests of
nonintersecting lattice paths. In \cite{Schl1}, a noncommutative binomial
theorem in $\C_{a,b;q,p}[x,y]$ for the elliptic binomial coefficients and
several convolution formulae analogous to the Chu--Vandermonde convolution
formula were derived. Since these convolutions were equivalent to Frenkel
and Turaev's ${}_{10}V_9$ summation \cite{FT} (see also
\cite[Theorem 2.3.1.]{Ros}), this provided a combinatorial interpretation
of this summation formula, which is one of the main identities in the theory
of elliptic hypergeometric series and can be stated as follows.
\begin{proposition}[\sc Frenkel and Turaev's ${}_{10}V_9$ summation]\label{prop:ft}
	Let $n\in\N_0$ and $a,b,c,d,e, q,p\in\C$ with $|p|<1$.
	Then we have
	\begin{multline}
		\sum_{k=0}^n\frac{\theta(aq^{2k};p)}{\theta(a;p)}
		\frac{(a,b,c,d,e,q^{-n};q,p)_k}
		{(q,aq/b,aq/c,aq/d,aq/e,aq^{n+1};q,p)_k}q^k\\\label{propfteq}
		=
		\frac{(aq,aq/bc,aq/bd,aq/cd;q,p)_n}
		{(aq/b,aq/c,aq/d,aq/bcd;q,p)_n},
	\end{multline}
	where $a^2q^{n+1}=bcde$.
\end{proposition}

The convolution formulae in \cite{Schl1}
were already extended to complex values by analytic continuation. In this
section we extend the binomial theorem to negative values of $n$ and $k$
and derive reflection and convolution formulae for the elliptic binomial coefficients. Since we proved the weight-dependent convolution formulae combinatorially in Section~\ref{sec:ref_comb}, we extend the combinatorial proof of the second author in \cite{Schl1} of Proposition~\ref{prop:ft} to negative values.
	
\subsection{Reflection formulae for elliptic binomial coefficients}
In \eqref{eqn:inv} we defined several involutions in the algebra
$\C_w[x,x^{-1},y,y^{-1}]$. Some of these involutions turn out to be
particularly nice in the elliptic case. Recall that 
\begin{align*}
	\widehat{w}(s,t)&=w(t,s)^{-1}, \\
	\widetilde{w}(s,t)&=w(1-s-t,t)^{-1}, \\
	\breve{w}(s,t)&=w(s,1-s-t)^{-1}.
\end{align*}
For the elliptic small weights we obtain
\begin{subequations}\label{eqn:ellinv}
	\begin{align}
	\widehat{w}_{a,b;q,p}(s,t)&=w_{b,a;q,p}(s,t)\label{ellinvhat} \\
	\widetilde{w}_{a,b;q,p}(s,t)&=w_{a/b,1/b;q,p}(s,t) \label{ellinvtilde}\\
	\breve{w}_{a,b;q,p}(s,t)&=w_{1/a,b/a;q,p}(s,t) \label{ellinvbreve},
	\end{align}
\end{subequations}
which follows from \eqref{eqn:ellinversion}.
Equation \eqref{ellinvtilde} was recently introduced by the second and
third author in \cite{SY}. We are now ready to derive the following
reflection formulae from Theorem~\ref{thm:refl1}~and~\ref{thm:refl2}
\begin{corollary}\label{cor:ellrefl}
Let $n,k \in \Z$. Then,
\begin{align*}
  \eBinom{n}{k}{a,b;q,p}
  &=\eBinom{n}{n-k}{b,a;q,p} \prod_{j=1}^{k} W_{a,b;q,p}(j,n-k)  \\
  &= (-1)^{n-k} \mathrm{sgn} (n-k) \eBinom{-k-1}{-n-1}{a/b,1/b;q,p}
    \prod_{j=1}^{n-k} W_{a,b;q,p}(n+1-j,j)^{-1} \\
  &= (-1)^{k} \mathrm{sgn} (k) \eBinom{k-n-1}{k}{1/a,b/a;q,p}
    \prod_{j=1}^{k} W_{a,b;q,p}(j,-j) ,
\end{align*}
where $\mathrm{sgn}(k)$ is defined by \eqref{def:sgn}.
\end{corollary}
\subsection{Elliptic binomial theorem}
Recall the definition of $\C_{a,b;q,p}[x,y]$ from
Definition~\ref{def:Cwxy3ell}. We will also consider the algebra of formal
power series  $\C_{a,b;q,p}[[x,y,y^{-1}]]$ and $\C_{a,b;q,p}[[x,x^{-1},y]]$.
Then, the weight-dependent noncommutative binomial theorem
(Theorem~\ref{thm:extbinomial}) reduces to the following theorem.
\begin{theorem}\label{thm:ellextbinomial}
Let $n,k\in \Z$. Then we have
\begin{equation}\label{eqn:ellextbinomial}
[x^k y^{n-k}](x+y)^n = \eBinom{n}{k}{a,b;q,p}.
\end{equation}
In particular, we have the expansions
\begin{align}
  (x+y)^n &= \sum_{k\geq0} \eBinom{n}{k}{a,b;q,p} x^k y^{n-k} \quad
            \text{in $\C_{a,b;q,p}[[x,y,y^{-1}]]$} \label{eqn:ellextbin1} \\
  \intertext{or}
  (x+y)^n &= \sum_{k\leq n} \eBinom{n}{k}{a,b;q,p} x^{k} y^{n-k} \quad
            \text{in $\C_{a,b;q,p}[[x,x^{-1},y]]$.} \label{eqn:ellextbin2}
\end{align}
\end{theorem}
The theorem reduces to the noncommutative $q$-binomial theorem by
Formichella and Straub \cite{FS} if we formally let $p\to 0$, $a\to 0$,
then $b\to 0$ (in this order). 

In \cite{Schl1} the elliptic case of the convolution formula in
Corollary~\ref{cor:wdconv1} was already derived for nonnegative $n$ and $m$, but the restriction to nonnegative values was removed by analytic continuation. In the following Corollary, $n$ and $m$ can also be negative integers.
\begin{corollary}\label{cor:ellconv1}
	Let $n,m \in \Z$, $k \geq 0$ and $a,b,q,p \in \C$ with $|p|<1$. Then we have  
	\begin{equation}\label{eqn:ellconv1id}
		\eBinom{n+m}{k}{a,b;q,p}=\sum_{j=0}^{k}
		\eBinom{n}{j}{a,b;q,p}
		\eBinom{m}{k-j}{aq^{2n-j},bq^{n+j};q,p}
		\prod_{i=1}^{k-j}W_{a,b;q,p}(i+j,n-j),
	\end{equation}
	where the elliptic binomial coefficients and $W_{a,b;q,p}(s,t)$ are defined
	by \eqref{def:ellcoef} and \eqref{def:ellbigweights}.
\end{corollary}
By changing the summation index $k$ to $j$, substituting the parameters $(a,b,c,d,e,n)$ appearing in
Equation~\eqref{propfteq} by
$(bq^{-n}/a,q^{-n}/a,bq^{1+n+m},bq^{-n-m+k}/a,q^{-n},k)$ and doing some basic manipulations one obtains Corollary~\ref{cor:ellconv1}. As described in \cite{Schl1}, this substitution is reversible if $q^{-n}$ and $q^{-m}$
are treated as complex variables. Therefore, Corollary~\ref{cor:ellconv1} is equivalent to Frenkel
and Turaev's ${}_{10}V_9$ summation~\eqref{propfteq} and we provided a combinatorial proof for $n,m \in \Z$ and $k \geq 0$ in Section~\ref{sec:ref_comb}.

\begin{corollary}\label{cor:ellconv2}
Let $n,m \in \Z$, $k \leq n+m$ and $a,b,q,p \in \C$ with $|p|<1$. Then we have  
 \begin{equation}\label{eqn:ellconv2id}
\eBinom{n+m}{k}{a,b;q,p}=\sum_{j=k-m}^{n}
\eBinom{n}{j}{a,b;q,p}
\eBinom{m}{k-j}{aq^{2n-j},bq^{n+j};q,p}
\prod_{i=1}^{k-j}W_{a,b;q,p}(i+j,n-j),
\end{equation}
where the elliptic binomial coefficients and $W_{a,b;q,p}(s,t)$ are defined
by \eqref{def:ellcoef} and \eqref{def:ellbigweights}.
\end{corollary}
As worked out in Section~\ref{sec:conv}, these two Convolution formulae are equivalent. Therefore, also Corollary~\ref{cor:ellconv2} is equivalent to Frenkel
and Turaev's ${}_{10}V_9$ summation~\eqref{propfteq}.
By repeated analytic continuation we can extend both corollaries to
complex $n$ and $m$.

\section{A matrix inversion}
In this section we will use the fact that Corollary~\ref{cor:wdconv1}
is true for all integers $n,m$ to obtain a weight-dependent matrix inversion.
Matrix inversions are an important tool in combinatorics and special
functions. They are important, for instance, in the theory of ordinary,
basic and elliptic hypergeometric functions \cite{KratMat,SchlMat,War}.
Given an infinite-dimensional lower-triangular matrix
$F=(f_{n,k})_{n,k \in \Z}$ with $f_{n,k}=0$ unless $n\geq k$, the matrix
$G=(g_{k,l})_{k,l \in \Z}$ is the inverse matrix of $F$ if and only if
\begin{equation*}
\sum_{k=l}^{n} f_{n,k} g_{k,l} = \delta_{n,l} \quad \text{for all $n,l \in \Z$}.
\end{equation*}
For convenience, in this section sums of the form $\sum_{j=u}^{v}$ are defined to be $0$ if $v \leq u-1$.
\begin{theorem}
Let $(w(s,t))_{s,t\in\Z},x,y$ be noncommuting variables as in
\eqref{eqn:noncommrel3} and $m$ be some integer.
If $F=(f_{n,k})_{n,k \in \Z}$ with
\begin{subequations}
\begin{align}
  f_{n,k}&= x^ky^{-m-k-1} \wBinom{m+n}{n-k}{w} y^{m+k+1} x^{-k}
           \prod_{i=1}^{n-k} W(i+k,-m-k-1),\\
\intertext{and $f_{n,k}=0$ if $k>n$, then $G=(g_{k,l})_{k,l \in \Z}$ with}
g_{k,l}&=x^l \wBinom{-m-l-1}{k-l}{w} x^{-l}	
\end{align}
and $g_{k,l} = 0$ if $l>k$ is the inverse matrix of $F$. 
\end{subequations}
\end{theorem}
Both $f_{n,k}$ and $g_{k,l}$ are formal in $\C[(w(s,t))_{s,t\in\Z}]$, since
they contain noncommuting variables $x$ and $y$, which can be understood
to be shift operators that cancel after shifting the weights $w(s,t)$.
\begin{proof}
We begin with Corollary~\ref{cor:wdconv1} and substitute
$(n,m,k) \mapsto (-m-l-1,m+n,n-l)$ for some $m,l,n \in \Z$ and $l\leq n$
to obtain that $\wBinomText{n-l-1}{n-l}{w}$ is equal to
\begin{align*}
\sum_{k=0}^{n-l}&\Bigg(\wBinom{-m-l-1}{k}{w}
\left(x^ky^{-m-l-k-1}\,\wBinom{m+n}{n-l-k}{w}y^{m+l+k+1}x^{-k}\right)\\
&\quad\times\prod_{i=1}^{n-l-k}W(i+k,-m-l-k-1)\Bigg)\\
&\quad=\sum_{k=l}^{n}\Bigg(\left(x^{k-l}y^{-m-k-1}\,
  \wBinom{m+n}{n-k}{w}y^{m+k+1}x^{-k+l}\right)\\
&\qquad\qquad\times\prod_{i=1}^{n-k}W(i+k-l,-m-k-1)\wBinom{-m-l-1}{k-l}{w}\Bigg).
\end{align*}
Multiplying $x^l$ from the left and $x^{-l}$ from the right yields
\begin{align*}	 		
\sum_{k=l}^{n} f_{n,k} g_{k,l} = x^l \wBinom{n-l-1}{n-l}{w} x^{-l} = \delta_{n,l}
\end{align*}
for all $l\leq n \in \Z$. For $l>n$, the sum is defined to be $0$,
which completes the proof.
\end{proof}

\bibliographystyle{plain}
\bibliography{bibliography}

\end{document}